\documentclass[11pt,a4wide,reqno]{article}
\usepackage{graphics}
\usepackage[pdftex]{graphicx}
\usepackage{amsmath}
\usepackage{amssymb,latexsym}
\usepackage[only,fivrm]{rawfonts}
\usepackage{bbm}
\usepackage{enumerate}
\usepackage{amsthm}
\usepackage{multirow}
\usepackage{nccmath}
\usepackage{indentfirst}
\usepackage{mathabx}
\usepackage{multirow}

\numberwithin{equation}{section}

\newcommand{\hb}{\hat{\beta}_n}

\newcommand{\hoMLE}{\hat{\lambda}_n}

\def\eqd{\,{\buildrel d \over =}\,}
\def\law{\xrightarrow{d}}
\def\hlr{\hat{\lambda}^R}
\def\hll{\hat{\lambda}^L}
\def\hlo{\hat{\lambda}^0}
\def\R{\mathbbm{R}}
\def\hl{\hat{\lambda}}

\newtheorem{theorem}{THEOREM}[section]
\newtheorem{lemma}{LEMMA}[section]
\newtheorem{remark}{REMARK}[section]

\def\argmin{\mathop{\text{\sl\em argmin}}}
\def\argmax{\mathop{\text{\sl\em argmax}}}
\renewcommand{\O}{\ensuremath{{\cal O}}}

\usepackage{color}

\begin{document}

\title{A likelihood ratio test for monotone baseline hazard functions in the Cox model}
\author {Gabriela F.~Nane\\[3mm]
Delft Institute of Applied Mathematics, Delft University of Technology}
\date{\today}
\maketitle

\begin{abstract}
We consider a likelihood ratio method for testing whether a monotone baseline hazard function in the Cox model has a particular value at a fixed point. The characterization of the estimators involved is provided both in the nondecreasing and the nonincreasing setting. These characterizations facilitate the derivation of the asymptotic distribution of the likelihood ratio test, which is identical in the nondecreasing and in the nonincreasing case. The asymptotic distribution of the likelihood ratio test enables, via inversion, the construction of pointwise confidence intervals. Simulations show that these confidence intervals exhibit comparable coverage probabilities with the confidence intervals based on the asymptotic distribution of the nonparametric maximum likelihood estimator of a monotone baseline hazard function.

\bigskip


\bigskip
\end{abstract}

\section{Introduction}
\noindent

In survival analysis, using Cox proportional hazards model~\cite{cox:1972} is the typical choice to account for the effect of covariates on the lifetime distribution. Its attractiveness resides in its form, that allows for efficient estimation of the regression coefficient, while leaving the baseline distribution completely unspecified, see e.g.,~\cite{efron:1977,oakes:1977,slud:1982}. The regression coefficient estimator is the well-known maximum partial likelihood estimator~\cite{cox:1972,cox:1975}.  As a response to Cox's paper, Breslow~\cite{cox:1972} proposed a different approach, that yields the same maximum partial likelihood estimator, along with an estimator for the baseline cumulative hazard function $\Lambda_0$. Impressive amount of research rapidly followed Cox's seminal paper, which primarily focused on deriving the (asymptotic) properties of the maximum partial likelihood estimator of the regression coefficient $\hb$, as well as of the Breslow estimator $\Lambda_n$ of the baseline cumulative hazard function.

Even though the baseline hazard $\lambda_0$ can be left completely unspecified, in practice, one might be interested in restricting $\lambda_0$ qualitatively. This can be done by assuming the baseline hazard to be monotone, for example, as suggested by Cox himself~\cite{cox:1972}. Various studies have indicated that a monotonicity constraint should be imposed on the baseline hazard, which complies with the medical expertise. For an illustration of a nonincreasing baseline hazard estimator in the study of patients with acute coronary syndrome, see~\cite{nanetal:2012}. Lopuha\"a and Nane~\cite{lopuhaa_nane1} propose a nonparametric maximum likelihood estimator and a Grenander type estimator for estimating a monotone baseline hazard function. The Grenander type estimator is defined in terms of slopes of the greatest convex minorant of the Breslow estimator $\Lambda_n$. The two estimators have been proven to be strongly consistent and have been shown to exhibit the same distributional law. Furthermore, at a fixed point $x_0$, the scaled difference between the maximum likelihood estimator $\hoMLE(x_0)$ and the true baseline hazard $\lambda_0(x_0)$ converges in distribution to the distribution of the minimum of two-sided Brownian motion plus a parabola times a constant depending on the underlying parameters. These results adhere to the general nonparametric shape constrained theory, and, in particular, prolong naturally the findings in the case of the random censorship model with no covariates~\cite{huangwellner:1995}.

Ensuing inference will be pursued in this paper, by testing the hypothesis that the underlying monotone baseline hazard has a particular value $\theta_0$, at a fixed point $x_0$. We will use a likelihood ratio test of $H_0: \lambda_0(x_0)=\theta_0$ versus $H_1: \lambda_0(x_0)\neq \theta_0$. Within the shape restricted problems, this approach was initially employed for monotone distributions in the current status model by Banerjee and Wellner~\cite{banerjeewellner:2001}. The authors focused on deriving the limiting distribution of the likelihood ratio test under the null hypothesis, and to obtain what the authors referred to a fixed universal distribution, defined in terms of slopes of the greatest convex minorant of the two-sided Brownian motion  plus a parabola. These findings were followed by a rapid stream of research, see, e.g.,~\cite{banerjeewellner:2005,banerjee:2007,banerjee:2008}, that revealed that the likelihood ratio method could be extended straightforwardly in other shape constrained settings. In this paper, we carry on this research for the monotone baseline hazard function in the Cox model. In addition to extending directly the results in the right censoring model with no covariates in~\cite{banerjee:2008}, we aim to provide a thorough description of the method and detailed proofs for all results.

Furthermore, we will derive confidence sets for $\lambda_0(x_0)$, based on the likelihood ratio method. More specifically, we will use that inverting the family of tests can yield, in turn, pointwise confidence intervals for the baseline hazard function. A more direct method of constructing pointwise confidence intervals is based on the asymptotic distribution, at a fixed point $x_0$, of the nonparametric maximum likelihood estimator $\hoMLE(x_0)$, derived in~\cite{lopuhaa_nane1}. Nonetheless, this entails the bothersome issue of estimating the nuisance parameter, and more specifically, estimating the derivative of the baseline hazard function $\lambda'(x_0)$, since, to the author's best knowledge, there is no available smooth monotone estimator of the baseline hazard function in the Cox model. One option would be to kernel smooth the NPMLE $\hoMLE$, but this would pose extra difficulties, like an appropriate choice of a bandwidth. For a discussion of this issues in the case of right-censoring with no covariates, see~\cite{banerjee:2008}.

The paper is organized as follows. Section~\ref{sec:def} introduces the Cox model, the notations and the common assumptions. In Section~\ref{sec:loglik}, we introduce the likelihood ratio method and characterize the maximum likelihood estimator $\hoMLE$ of a monotone baseline hazard function and the estimator $\hoMLE^0$, for which $\hoMLE^0(x_0)=\theta_0$, for a fixed $x_0$ in the interior of the support of the baseline distribution. We provide the characterization of the two estimators in the case of both nondecreasing and nonincreasing baseline hazard functions $\lambda_0$. The asymptotic distribution of the likelihood ratio statistic is provided, along with preparatory lemmas, in Section~\ref{sec:limit_distr}. Finally, Section~\ref{sec:simulations} is devoted to constructing pointwise confidence intervals and comparing them, via simulations, with the conventional confidence intervals based on the asymptotic distribution of the NPMLE $\hoMLE$.\\


\section{Definitions and assumptions}
\label{sec:def}
\noindent
Suppose that the observed data consist of the following independent and identically distributed triplets $(T_i,\Delta_i,Z_i)$, with $i=1,\ldots,n$. The event time, denoted by~$X$ and commonly referred to as the survival time is subject to random censoring. Thus, $T=\min(X,C)$, where~$T$ is the follow-up time and~$C$ denotes the censoring time. The indicator~$\Delta=\{X\leq C\}$ marks whether the follow-up time is an event or a censoring time. Finally,~$Z\in\R^p$ denotes the covariate vector of the observed follow-up time $T$, which is assumed to be time invariant. The event time~$X$ and censoring time~$C$ are assumed to be conditionally independent, given the covariate vector~$Z$. Furthermore, let~$F$ be the distribution function of the non-negative random variable~$X$,~$G$ the distribution function of the non-negative random variable~$C$, and~$H$ the distribution function of~$T$. The distribution function~$F(x|z)$ is assumed to be absolutely continuous, with density~$f(x|z)$. Similarly, the distribution function $G(c|z)$ is assumed to be absolutely continuous, with density~$g(c|z)$. In addition,~$F(x|z)$ and~$G(c|z)$ share no parameters, thus the censoring mechanism is assumed to be non-informative.

Let $\lambda(x|z)$ be the hazard function for an individual with covariate vector $z\in\R^p$. The Cox model~\cite{cox:1972} specifies that

\begin{equation}
\label{def:cox model}
\lambda\left(x|z\right)=\lambda_0(x)\,\text{e}^{\beta_0' z},
\end{equation}
where $\lambda_0$ represents the baseline hazard function, that corresponds to $z=0$, and $\beta_0\in\R^p$
is the vector of the underlying regression coefficients. Finally, we consider the following assumptions, that are typically employed when deriving large sample properties of estimators within the Cox model (e.g., see~\cite{tsiatis:1981}).
\begin{description}
\item (A1)
Let $\tau_F,\tau_G$ and $\tau_H$ be the end points of the support of $F,G$ and $H$ respectively.
Then
\[
\tau_{H}=\tau_G<\tau_F\leq \infty.
\]
\item (A2)
There exists $\varepsilon>0$ such that
\[
\sup_{|\beta-\beta_0 |\leq \varepsilon} \mathbbm{E}\left[ |Z|^2\, \text{e}^{2\beta' Z}\right]<\infty,
\]
where $|\cdot|$ denotes the Euclidean norm.
\end{description}





\section{The likelihood ratio method and the characterization of the estimators}
\label{sec:loglik}
\noindent
By definition, $\Lambda(x|z)=-\log (1-F(x|z))$ is the cumulative hazard function.
Thus, from~\eqref{def:cox model}, it follows that~$\Lambda(x|z)=\Lambda_0(x)\exp(\beta_0'z)$, where~$\Lambda_0(x)=\int_0^x \lambda_0(u)\,\mathrm{d}u$ is the
baseline cumulative hazard function. Since, for a continuous distribution, $\lambda(t)=f(t)/(1-F(t))$, for $t\geq0$, the full likelihood is given by

\[
\begin{split}
&
\prod_{i=1}^n
\left\{f(T_i\mid Z_i)\left[1-G(T_i\mid Z_i)\right]\right\}^{\Delta_i}
\left\{g(T_i\mid Z_i)\left[1-F(T_i\mid Z_i)\right]\right\}^{1-\Delta_i}\\
&=
\prod_{i=1}^n
\lambda(T_i\mid Z_i)^{\Delta_i}\exp\left[-\Lambda(T_i\mid Z_i)\right]
\times
\prod_{i=1}^n
\left[1-G(T_i\mid Z_i)\right]^{\Delta_i}
g(T_i\mid Z_i)^{1-\Delta_i}.
\end{split}
\]
As the censoring mechanism is assumed to be non-informative, and by~\eqref{def:cox model}, maximizing the full likelihood is the same as maximizing

\[
\prod_{i=1}^n
\lambda(T_i\mid Z_i)^{\Delta_i}\exp\left[-\Lambda(T_i\mid Z_i)\right]
=
\prod_{i=1}^n
\left[\lambda_0(T_i)\mathrm{e}^{\beta_0'Z_i}\right]^{\Delta_i}
\exp\left[-\mathrm{e}^{\beta_0'Z_i}\Lambda_0(T_i)\right],
\]
which yields the following (pseudo) loglikelihood function, written as function of $\beta\in\mathbbm{R}^p$ and $\lambda_0$

\[
\sum_{i=1}^n\left[
\Delta_{i}\log\lambda_0(T_{i}) +\Delta_{i}\beta' Z_{i}-\text{e}^{\beta' Z_{i}}\Lambda_0(T_{i})\right].
\]
Let $T_{(1)}<T_{(2)}<\cdots<T_{(n)}$ be the ordered follow-up times and,
for~$i=1,\ldots,n$, let $\Delta_{(i)}$ and~$Z_{(i)}$ be the censoring indicator and covariate vector corresponding to $T_{(i)}$. Writing the above (pseudo) likelihood as a function of $\beta$ and $\lambda_0$ gives

\begin{equation}
\label{likelihood}
L_{\beta}(\lambda_0)=\sum_{i=1}^n\left[
\Delta_{(i)}\log\lambda_0(T_{(i)}) +\Delta_{(i)}\beta' Z_{(i)}-\text{e}^{\beta' Z_{(i)}}\int_0^{T_{(i)}}\lambda_0(u)\,\mathrm{d}u
\right].
\end{equation}
Following the approach in~\cite{lopuhaa_nane1}, we do not proceed with the joint maximization of \eqref{likelihood} over~$\beta$ and monotone~$\lambda_0$. Alternatively, for $\beta\in\R^p$ fixed, we consider maximum likelihood estimation of a monotone baseline hazard function $\lambda_0$ and denote the estimator by $\hoMLE(x;\beta)$. Afterwards, we simply replace~$\beta$ by~$\hb$, the maximum partial likelihood estimator ( see, e.g.,~\cite{cox:1972,cox:1975}) of the underlying regression coefficients~$\beta_0$, due to its commendable asymptotic properties (see, e.g.,~\cite{efron:1977,oakes:1977,slud:1982}). The proposed NPMLE is thus $\hoMLE(x)=\hoMLE(x;\hb)$ and will be referred to as the unconstrained estimator of a monotone $\lambda_0$. Furthermore, for $\beta\in\R^p$ fixed, we maximize the loglikelihood function $L_{\beta}(\lambda_0)$ in~\eqref{likelihood} over the class of all monotone baseline hazard functions, under the null hypothesis $H_0:\lambda_0(x_0)=\theta_0$, for $x_0\in(0,\tau_H)$ and $\theta_0\in(0,\infty)$, fixed. We obtain $\hoMLE^0(x;\beta)$ and hence propose $\hoMLE^0(x)=\hoMLE^0(x;\hb)$ as the constrained NPMLE.

Replacing $\beta$ by $\hb$ also in the loglikelihood function~\eqref{likelihood} yields the likelihood ratio statistic for testing $H_0:\lambda_0(x_0)=\theta_0$,

\begin{equation}
\label{likelihood_ratio}
2\log \xi_n(\theta_0)=2 L_{\hb}(\hoMLE)-2 L_{\hb}(\hoMLE^0).
\end{equation}
Thus, for computing the likelihood ratio statistic, we need to characterize the unconstrained NPMLE $\hoMLE(x)$ and the constrained NPMLE  $\hoMLE^0(x)$ of a monotone baseline hazard function $\lambda_0$.

\subsection{Nondecreasing baseline hazard}
\label{sec:nondecreasing}
We consider first maximum likelihood estimation of a nondecreasing baseline hazard function $\lambda_0$. Both the unconstrained estimator $\hoMLE$ and the constrained estimator $\hoMLE^0$ will be characterized in terms of the processes

\begin{equation}
\label{def:W_n}
W_n(\beta,x)
=
\int \left( \text{e}^{\beta' z} \int_0^x \{ u\geq s \}\,\mathrm{d}s \right)\, \mathrm{d}P_n(u,\delta,z),
\end{equation}
and
\begin{equation}
\label{def:V_n}
V_n(x)
=
\int
\delta\{u< x\}\,\mathrm{d}P_n(u,\delta,z),
\end{equation}
with $\beta\in\R^p$ and $x\geq 0$, and where $P_n$ is the empirical measure of
the $(T_i,\Delta_i,Z_i)$, with $i=1,\ldots,n$. The characterization of the unconstrained estimator $\hoMLE(x;\beta)$ has already been provided in Lemma 1 in~\cite{lopuhaa_nane1}, which we restate below. Furthermore, we provide a closed form of the estimator on blocks of indices on which the estimator is constant. This expression will be useful in deriving the asymptotic distribution of the likelihood ratio statistic.

\begin{lemma}
\label{lemma:characterization_unconstrained}
Let $T_{(1)}<\ldots<T_{(n)}$ be the ordered follow-up times and consider a fixed $\beta\in\R^p$.
\begin{enumerate}[(i)]
\item
Let $W_n$ and $V_n$ defined in~\eqref{def:W_n} and~\eqref{def:V_n}.
Then, the NPMLE $\hoMLE(x;\beta)$ of a nondecreasing baseline hazard function $\lambda_0$ is of the form
\[
\hoMLE(x;\beta)
=
\begin{cases}
0 & x < T_{(1)},\\
\hl_i & T_{(i)}\leq x < T_{(i+1)}, \text{ for }i=1,\ldots,n-1,\\
\infty & x\geq T_{(n)},
\end{cases}
\]
where $\hat{\lambda}_i$ is the left derivative of the greatest convex minorant (GCM) at the point $P_i$
of the cumulative sum diagram (CSD) consisting of the points

\begin{equation}
\label{def:points_Pj}
P_j=\Big( W_n(\beta,T_{(j+1)})-W_n(\beta,T_{(1)}), V_n(T_{(j+1)}) \Big),
\end{equation}
for $j=1,\ldots,n-1$ and $P_0=(0,0)$.
\item
For $k\geq 1$, let $B_1,\ldots,B_k$ be blocks of indices such that $\hoMLE(x;\beta)$ is constant on each block and $B_1\cup \ldots\cup B_k=\{1,\ldots,n-1\}$. Denote by $v_{nj}(\beta)$ the value of $\hoMLE(x;\beta)$ on block $B_j$. Then,

\begin{equation}
\label{def:v_nj}
v_{nj}(\beta)=\frac{\sum_{i\in B_j}\Delta_{(i)}}{\sum_{i\in B_j}\left[T_{(i+1)}-T_{(i)}\right]\sum_{l=i+1}^{n}\text{e}^{\beta'Z_{(l)}}}.
\end{equation}
\end{enumerate}
\end{lemma}

\begin{proof}
The proof of (i) has been provided by Lemma 1 in~\cite{lopuhaa_nane1}. The NPMLE $\hoMLE(x;\beta)$ is obtained by maximizing the (pseudo) loglikelihood function in~\eqref{likelihood} over all $0\leq\lambda_0(T_{(1)})\leq\ldots\leq\lambda_0(T_{(n)})$. As argued in~\cite{lopuhaa_nane1}, the estimator has to be a nondecreasing step function, that is zero for $x<T_{(1)}$, constant on the interval $[T_{(i)},T_{(i+1)})$, for $i=1,\ldots,n-1$ and can be chosen arbitrarily large for $x\geq T_{(n)}$. Then, for fixed $\beta\in\R^p$, the (pseudo) loglikelihood function in~\eqref{likelihood} reduces to

\begin{equation}
\label{likelihood_inc_reduced}
\begin{split}
\sum_{i=1}^{n-1}&
\Delta_{(i)}\log\lambda_0(T_{(i)})
-
\sum_{i=2}^n
\text{e}^{\beta' Z_{(i)}}
\sum_{j=1}^{i-1}
\left[ T_{(j+1)}-T_{(j)}\right]
\lambda_0(T_{(j)})\\
&=
\sum_{i=1}^{n-1}
\left\{
\Delta_{(i)}\log\lambda_0(T_{(i)})
-
\lambda_0(T_{(i)})
\left[ T_{(i+1)}-T_{(i)}\right]
\sum_{l=i+1}^n
\text{e}^{\beta' Z_{(l)}}
\right\}.
\end{split}
\end{equation}
Let $\lambda_i=\lambda_0(T_{(i)})$, for $i=1,\ldots,n-1$ and $\lambda=(\lambda_1,\ldots,\lambda_{n-1})$. Then, finding the NPMLE reduces to maximizing

\begin{equation}
\label{likelihood1}
\phi(\lambda)
=
\sum_{i=1}^{n-1}
\left\{
\Delta_{(i)}\log\lambda_i
-
\lambda_i
\left[ T_{(i+1)}-T_{(i)}\right]
\sum_{l=i+1}^n
\text{e}^{\beta' Z_{(l)}}
\right\},
\end{equation}
over the set $0\leq\lambda_1\leq\ldots\leq\lambda_{n-1}$. The NPMLE corresponds thus to a vector $\hl=(\hl_1,\ldots,\hl_{n-1})$ that maximizes $\phi$ over $0\leq\lambda_1\leq\ldots\leq\lambda_{n-1}$.
To prove (ii), we first derive the Fenchel conditions of the estimator. Thus, we will show that the estimator $\hoMLE(x;\beta)$ maximizes the (pseudo) loglikelihood function in~\eqref{likelihood} over the class of nondecreasing baseline hazard functions if and only if

\begin{equation}
\label{cond:fenchel1}
\sum_{j\geq i}\left\{ \frac{\Delta_{(j)}}{\hl_j}- \left[ T_{(j+1)}-T_{(j)}\right]
\sum_{l=j+1}^{n}
\text{e}^{\beta' Z_{(l)}}\right\}\leq 0,
\end{equation}
for $i=1,\ldots,n-1$, and

\begin{equation}
\label{cond:fenchel2}
\sum_{j=1}^{n-1}\left\{ \frac{\Delta_{(j)}}{\hl_j}- \left[ T_{(j+1)}-T_{(j)}\right]
\sum_{l=j+1}^{n}
\text{e}^{\beta' Z_{(l)}}\right\}\hl_j=0.
\end{equation}
The NPMLE $\hoMLE(x;\beta)$ is thus uniquely determined by these Fenchel conditions.
The rest of the proof focuses on deriving the Fenchel conditions \eqref{cond:fenchel1} and \eqref{cond:fenchel2} and on establishing~\eqref{def:v_nj}.

First, note that the function $\phi$ in~\eqref{likelihood1} is concave and that the vector of partial derivatives $\nabla\phi(\lambda)=(\nabla_1\phi(\lambda),\ldots,\nabla_{n-1}\phi(\lambda))$ is given by

\[
\nabla\phi(\lambda)
=\left(
\frac{\Delta_{(1)}}{\lambda_1}-\left[T_{(2)}-T_{(1)}\right]\sum_{l=2}^{n}\text{e}^{\beta'Z_{(l)}},
\ldots,\frac{\Delta_{(n-1)}}{{\lambda_{n-1}}}-\left[T_{(n)}-T_{(n-1)}\right]\text{e}^{\beta'Z_{(n)}}
\right).
\]
%
%
Define now the functions $g_i(\lambda)=\lambda_{i-1}-\lambda_{i}$, for $i=1,\ldots,n-1$ and $\lambda_0=0$, and the vector $g(\lambda)=(g_1(\lambda),\ldots,g_{n-1}(\lambda))$. Moreover, define the matrix of partial derivatives by

\begin{equation}
\label{def:matrix_par_der}
G=\left(\frac{\partial g_i(\lambda)}{\partial\lambda_j}  \right), \quad\text{for } i=1,\ldots,n-1;\,j=1,\ldots,n-1.
\end{equation}
Let $\tilde{\phi}(\lambda)=-\phi(\lambda)$. Then, maximizing~\eqref{likelihood1} over all $0\leq\lambda_1\leq\ldots\leq\lambda_{n-1}$ is equivalent with minimizing $\tilde{\phi}(\lambda)$ under the restriction that all components of the vector $g(\lambda)$ are negative.
An adaptation of the Karush-Kuhn-Tucker theorem (e.g., see Theorem 8.1 in~\cite{groeneboom_course}) states that~$\hl$ minimizes $\tilde{\phi}$ over all vectors $\lambda$ such that $g_i(\lambda)\leq0$, for all $i=1,\ldots,n-1$, if and only if the following conditions hold

\begin{equation}
\label{cond:KKT1}
\nabla\tilde{\phi}(\hl)+G^{T}\alpha=0,
\end{equation}

\begin{equation}
\label{cond:KKT2}
g(\hl)+w=0,
\end{equation}

\begin{equation}
\label{cond:KKT3}
\langle\alpha,w\rangle=0,
\end{equation}
for $\alpha=(\alpha_1,\ldots,\alpha_{n-1})$, with $\alpha_i\geq0$, $i=1,\ldots,n-1$ and $w=(w_1,\ldots,w_{n-1})$, with $w_i\geq0$, for $i=1,\ldots,n-1$. The first condition~\eqref{cond:KKT1}, yields that

\begin{equation}
\label{alpha_i}
\alpha_i=-\sum_{j\geq i}\nabla_j\phi(\hl)=-\sum_{j\geq i}\left\{ \frac{\Delta_{(j)}}{\hl_j}- \left[ T_{(j+1)}-T_{(j)}\right]
\sum_{l=j+1}^{n}
\text{e}^{\beta' Z_{(l)}}\right\}.
\end{equation}
Since $\alpha_i\geq0$, for all $i=1,\ldots,n-1$, condition~\eqref{cond:fenchel1} is immediate.
From~\eqref{cond:KKT2}, $w=-g(\hl)=(\hl_1-\hl_0,\ldots,\hl_{n-1}-\hl_{n-2})$, with $\hl_0=0$.
Note that the condition $w_i\geq0$ implies that $\hl_{i-1}\leq\hl_{i}$, for all $i=1,\ldots,n-1$, which is trivially satisfied.
Finally, by~\eqref{cond:KKT3},

\[
\sum_{i=1}^{n-1}(\hl_{i}-\hl_{i-1})\sum_{j\geq i}\nabla_j\phi(\hl)=0,
\]
which re-writes exactly to~\eqref{cond:fenchel2}.

To derive the expression in~\eqref{def:v_nj}, we prove first that~\eqref{cond:fenchel1} and~\eqref{cond:fenchel2} imply that

\begin{equation}
\label{cond:extra}
\sum_{j=1}^{n-1}\left\{ \frac{\Delta_{(j)}}{\hl_j}- \left[ T_{(j+1)}-T_{(j)}\right]
\sum_{l=j+1}^{n}
\text{e}^{\beta' Z_{(l)}}\right\}=0.
\end{equation}
Condition~\eqref{cond:fenchel1} gives that $\sum_{j=1}^{n-1}\nabla_j\phi(\hl)\leq0$. In addition, as the maximizer $\hl$ is nondecreasing,

\[
\begin{split}
\hl_1\sum_{j=1}^{n-1}\nabla_j\phi(\hl)=&-\nabla_2\phi(\hl)\hl_2-\nabla_3\phi(\hl)\hl_3-\ldots-\nabla_{n-1}\phi(\hl)\hl_{n-1}\\
&+ \nabla_2\phi(\hl)\hl_1+\nabla_3\phi(\hl)\hl_1+\ldots+\nabla_{n-1}\phi(\hl)\hl_1\\
=&\sum_{i=2}^{n-1}(\hl_{i-1}-\hl_i)\sum_{j\geq i}\nabla_j\phi(\hl)\geq 0.
\end{split}
\]
This shows~\eqref{cond:extra}.
Now let $B_1,\ldots,B_k$ be blocks of indices on which $\hl$ is constant such that $B_1\cup\ldots\cup B_k=\{1,\ldots,n-1\}$ and let $v_{nj}(\beta)$ be the value of $\hl$ on the block $B_j$, with $j=1,\ldots,k$. If $k=1$, then the expression of $v_{n1}$ is immediate from~\eqref{cond:extra}. Moreover, observe that, by~\eqref{cond:KKT3}, $\sum_{i=1}^{n-1}\alpha_i\left(\hl_i-\hl_{i-1}\right)=0$, and since $\alpha_i\geq0$ and $\hl_i\geq\hl_{i-1}$, for any $i=1,\ldots,n-1$, it will follow that $\alpha_i=0$, whenever $\hl_{i-1}<\hl_i$. Hence, for $k\geq2$, there exist $k-1$ $\alpha$'s that are zero. Then~\eqref{def:v_nj} follows by~\eqref{alpha_i} and~\eqref{cond:extra}. For example, for $k\geq3$, choose any two consecutive~$\alpha_i$ that are zero. From~\eqref{alpha_i}, we get that by subtracting these $\alpha_i$'s,

\[
\sum_{i\in B_j}\nabla_i\phi(\hl)= \sum_{i\in B_j}\left\{
\frac{\Delta_{(i)}}{v_{nj}(\beta)}-\left[T_{(i+1)}-T_{(i)}\right]\sum_{l=i+1}^{n}\text{e}^{\beta'Z_{(l)}}\right\}
= 0.
\]
As $v_{nj}(\beta)$ is constant on $B_j$, this yields~\eqref{def:v_nj}.
\end{proof}
As mentioned beforehand, the proposed unconstrained estimator is thus $\hoMLE(x)=\hoMLE(x;\hb)$. Equivalently, on each block of indices $B_j$, for $j=1,\ldots,k$, we propose the estimate $\hat{v}_{nj}=v_{nj}(\hb)$.
Under the null hypothesis $H_0:\lambda_0(x_0)=\theta_0$, the characterization of the constrained maximum likelihood estimator $\hoMLE^0$ is provided by the next lemma.

\begin{lemma}
\label{lemma:characterization_constrained}

Let $x_0\in(0,\tau_H)$ fixed, such that $T_{(m)}< x_0< T_{(m+1)}$, for a given $1\leq m\leq n-1$. Consider a fixed $\beta\in\R^p$.

\begin{enumerate}[(i)]
\item
For $i=1,\ldots,m$, let $\hll_i$ be the left derivative of the GCM at the point $P_i^L$ of the CSD consisting of the points $P_j^L=P_j$, for $j=1,\ldots,m$, with $P_j$ defined in~\eqref{def:points_Pj} and $P_0^L=(0,0)$. Moreover, for $i=m+1,\ldots,n-1$, let $\hlr_i$ be the left derivative of the GCM at the point $P_i^R$ of the CSD consisting of the points $P_j^R=P_j$, for $j=m,\ldots,n-1$, with $P_j$ defined in~\eqref{def:points_Pj}.
Then, for $\theta_0\in(0,\infty)$, the NPMLE $\hoMLE^0(x)$ of a nondecreasing baseline hazard function $\lambda_0$, under the null hypothesis $H_0:\lambda_0=\theta_0$, is of the form

\begin{equation}
\label{def:constrained_estim}
\hoMLE^0(x;\beta)
=
\begin{cases}
0 & x < T_{(1)},\\
\hlo_i & T_{(i)}\leq x < T_{(i+1)}, \text{ for }i\in\{1,\ldots,n-1\}\setminus{\{m\}}\\
\hlo_m & T_{(m)}\leq x < x_0,\\
\theta_0 & x_0\leq x<T_{(m+1)},\\
\infty & x\geq T_{(n)},
\end{cases}
\end{equation}
where $\hlo_i=\min(\hll_i,\theta_0)$, for $i=1,\ldots,m$, and $\hlo_i=\max(\hlr_i,\theta_0)$, for $i=m+1,\ldots,n-1$.

\item
For $k\geq1$, let $B_1^0,\ldots,B_k^0$ be blocks of indices such that $\hoMLE^0(x;\beta)$ is constant on each block and $B_1^0\cup\ldots\cup B_k^0=\{1,\ldots,n-1\}$. Then, there is one block, say $B_r^0$, on which $\hoMLE^0(x;\beta)$ is equal to $\theta_0$, and one block, say $B_p^0$, that contains $m$. On all other blocks $B_j^0$, denote by $v_{nj}^0(\beta)$ the value of $\hoMLE^0(x;\beta)$ on block $B_j^0$. Then,

\begin{equation}
\label{def:v_nj0}
v_{nj}^0(\beta)=\frac{\sum_{i\in B_j^0}\Delta_{(i)}}{\sum_{i\in B_j^0}\left[T_{(i+1)}-T_{(i)}\right]\sum_{l=i+1}^{n}\text{e}^{\beta'Z_{(l)}}},
\end{equation}
for $j=1,\ldots,p-1,p+1,\ldots,k$. On the block $B_p^0$ that contains $m$,

\begin{equation}
\label{def:v_np0}
v_{np}^0(\beta)=\frac{\sum_{i\in B_p^0}\Delta_{(i)}}{\sum_{i\in B_p^0\setminus\{m\}}\left[T_{(i+1)}-T_{(i)}\right]\sum_{l=i+1}^{n}\text{e}^{\beta'Z_{(l)}}+[x_0-T_{(m)}]\sum_{l=m+1}^{n}\text{e}^{\beta'Z_{(l)}}}.
\end{equation}
\end{enumerate}
\end{lemma}

\begin{proof}
We will derive the Karush-Kuhn-Tucker (KKT) conditions, that uniquely determine the constrained NPMLE, and which implicitly provide the characterization in (ii). To prove the lemma, we will show that the estimator proposed in (i) satisfies these conditions.

The constrained NPMLE estimator is obtained by maximizing the objective function~\eqref{likelihood} over $0\leq \lambda_0(T_{(1)})\leq\ldots\leq\lambda_0(T_{(m)})\leq\theta_0\leq\lambda_0(T_{(m+1)})\leq\ldots\leq \lambda_0(T_{(n-1)})$. In line with the reasoning for the unconstrained estimator, it can be argued that the constrained estimator has to be a nondecreasing step function that is zero for $x<T_{(1)}$, constant on $[T_{(i)},T_{(i+1)})$, for $i=1,\ldots,n-1$, is equal to $\theta_0$ on the interval $[x_0,T_{(m+1)})$, and can be chosen arbitrarily large for $x\geq T_{(n)}$. Therefore, for a fixed $\beta\in\R$, the (pseudo) loglikelihood function in~\eqref{likelihood} reduces to

\begin{equation}
\label{likelihood_inc_constr_reduced}
\begin{split}
&\sum_{i=1}^{m-1}
\left\{
\Delta_{(i)}\log\lambda_0(T_{(i)})
-
\lambda_0(T_{(i)})
\left[ T_{(i+1)}-T_{(i)}\right]
\sum_{l=i+1}^n
\text{e}^{\beta' Z_{(l)}}
\right\}\\
&\quad+
\Delta_{(m)}\log\lambda_0(T_{(m)})
-
\lambda_0(T_{(m)})
\left[ x _0-T_{(m)}\right]
\sum_{l=m+1}^n
\text{e}^{\beta' Z_{(l)}}\\
&\quad-
\theta_0
\left[ T_{(m+1)}-x _0\right]
\sum_{l=m+1}^n
\text{e}^{\beta' Z_{(l)}}\\
&\quad+
\sum_{i=m+1}^{n-1}
\left\{
\Delta_{(i)}\log\lambda_0(T_{(i)})
-
\lambda_0(T_{(i)})
\left[ T_{(i+1)}-T_{(i)}\right]
\sum_{l=i+1}^n
\text{e}^{\beta' Z_{(l)}}
\right\}.
\end{split}
\end{equation}
By letting $\lambda_i=\lambda_0(T_{(i)})$, for $i=1,\ldots,n-1$ and $\lambda=(\lambda_1,\ldots,\lambda_{n-1})$, we then want to maximize
\begin{equation}
\label{function_phi_constrained}
\begin{split}
\phi^0(\lambda)
=&
\sum_{i=1}^{m-1}
\left\{
\Delta_{(i)}\log\lambda_i
-
\lambda_i
\left[ T_{(i+1)}-T_{(i)}\right]
\sum_{l=i+1}^n
\text{e}^{\beta' Z_{(l)}}
\right\}\\
&+
\Delta_{(m)}\log\lambda_m
-
\lambda_m
\left[ x _0-T_{(m)}\right]
\sum_{l=m+1}^n
\text{e}^{\beta' Z_{(l)}}\\
&+
\sum_{i=m+1}^{n-1}
\left\{
\Delta_{(i)}\log\lambda_i
-
\lambda_i
\left[ T_{(i+1)}-T_{(i)}\right]
\sum_{l=i+1}^n
\text{e}^{\beta' Z_{(l)}}
\right\},
\end{split}
\end{equation}
over the set $0\leq\lambda_1\leq\ldots\leq\lambda_m\leq\theta_0\leq\lambda_{m+1}\leq\ldots\leq\lambda_{n-1}$. Let the vector $\hl^c=(\hl^c_1,\ldots,\hl^c_{n-1})$ denote the constrained NPMLE under the null hypothesis $H_0:\lambda_0(x_0)=\theta_0$.
We will show next that $\hl^c$ maximizes the objective function in~\eqref{function_phi_constrained} over the class of nondecreasing baseline hazard functions, under the null hypothesis, if and only if the following conditions are satisfied

\begin{equation}
\label{cond:fenchel1_constr}
\sum_{j\leq i}\left\{ \frac{\Delta_{(j)}}{\hl^c_j}- \left[ T_{(j+1)}-T_{(j)}\right]
\sum_{l=j+1}^{n}
\text{e}^{\beta' Z_{(l)}}\right\}\geq 0, \qquad \text{for } i=1,\ldots,m-1,
\end{equation}

\begin{equation}
\label{cond:fenchel1_constr_2}
\begin{split}
\sum_{j=1}^{m-1}\Bigg\{ \frac{\Delta_{(j)}}{\hl^c_j}- &\left[ T_{(j+1)}-T_{(j)}\right]
\sum_{l=j+1}^{n}
\text{e}^{\beta' Z_{(l)}}\Bigg\}\\
&+
\frac{\Delta_{(m)}}{\hl^c_m}- \left[ x_0-T_{(m)}\right]
\sum_{l=m+1}^{n}
\text{e}^{\beta' Z_{(l)}}
\geq 0,
\end{split}
\end{equation}

\begin{equation}
\label{cond:fenchel1_constr_1}
\sum_{j\geq i}\left\{ \frac{\Delta_{(j)}}{\hl^c_j}- \left[ T_{(j+1)}-T_{(j)}\right]
\sum_{l=j+1}^{n}
\text{e}^{\beta' Z_{(l)}}\right\}\leq 0, \qquad \text{for } i=m+1,\ldots,n-1,
\end{equation}
and

\begin{equation}
\label{cond:fenchel2_constr}
\begin{split}
\sum_{\substack{j=1\\  j\neq m}}^{n-1}\Bigg\{ \frac{\Delta_{(j)}}{\hl^c_j}-&\left[ T_{(j+1)}-T_{(j)}\right]
\sum_{l=j+1}^{n}
\text{e}^{\beta' Z_{(l)}}\Bigg\}\left( \hl^c_j-\theta_0\right)\\
&+
\left\{ \frac{\Delta_{(m)}}{\hl^c_m}-\left[ x_0-T_{(m)}\right]
\sum_{l=m+1}^{n}
\text{e}^{\beta' Z_{(l)}}\right\}\left( \hl^c_m-\theta_0\right)
=
0.
\end{split}
\end{equation}
The NPMLE $\hl^c$ is thus uniquely determined by these conditions. To prove (i), we will show that $\hoMLE^0$ defined in~\eqref{def:constrained_estim} verifies the Karush-Kuhn-Tucker (KKT) conditions \eqref{cond:fenchel1_constr}-\eqref{cond:fenchel2_constr}. Therefore,~$\hoMLE^0$ is the unique maximizer of $\phi^0(\lambda)$ in~\eqref{function_phi_constrained}, over the set $0\leq\lambda_1\leq\ldots\leq\lambda_m\leq\theta_0\leq\lambda_{m+1}\leq\ldots\leq\lambda_{n-1}$. As it will be seen further, despite bothersome calculations, the distinct form of the likelihood grants a unified framework for deriving the KKT conditions, that uses all the follow-up times, unlike the reasoning in~\cite{banerjeewellner:2001}, where the (pseudo) loglikelihood is split and arguments are carried both to the left and to the right of $x_0$.

Similar to the unconstrained case, observe that the function $\phi^0$ is concave and that the vector of partial derivatives is $\nabla\phi^0(\lambda)=(\nabla_1\phi^0(\lambda),\ldots,\nabla_{n-1}\phi^0(\lambda))$, with

\[
\nabla_i\phi^0(\lambda)
=
\frac{\Delta_{(i)}}{\lambda_i}-\left[T_{(i+1)}-T_{(i)}\right]\sum_{l=i+1}^{n}\text{e}^{\beta'Z_{(l)}},
\]
for $i=1,\ldots,m-1,m+1,\ldots,n-1$, and

\[
\nabla_m\phi^0(\lambda)
=
\frac{\Delta_{(m)}}{\lambda_m}-\left[x_0-T_{(m)}\right]\sum_{l=m+1}^{n}\text{e}^{\beta'Z_{(l)}}.
\]
Note that the form of $\nabla_m\phi^0(\lambda)$ differs from the form of $\nabla_i\phi^0(\lambda)$, for $i=1,\ldots,m-1,m+1,\ldots,n-1$.
Moreover, define the vector $g(\lambda)=(g_1(\lambda),\ldots,g_{n-1}(\lambda))$, with

\[
g_i(\lambda)
=
\begin{cases}
\lambda_i-\lambda_{i+1} & i =1,\ldots,m-1,\\
\lambda_m-\theta_0 & i=m,\\
\theta_0-\lambda_{m+1} & i=m+1,\\
\lambda_{i-1}-\lambda_i & i=m+2,\ldots,n-1,
\end{cases}
\]
and consider the matrix of partial derivatives defined in~\eqref{def:matrix_par_der}. Computations as in~\eqref{alpha_i} can be derived to show that condition~\eqref{cond:KKT1} yields~\eqref{cond:fenchel1_constr}-\eqref{cond:fenchel1_constr_1}, upon noting that

\begin{equation}
\label{alpha_i_constr}
\alpha_i
=
\begin{cases}
\sum_{j\leq i}\nabla_j\phi^0(\hl^c) & i=1,\ldots,m,\\
-\sum_{j\geq i}\nabla_j\phi^0(\hl^c) & i=m+1,\ldots,n-1.
\end{cases}
\end{equation}
Condition~\eqref{cond:KKT2} gives that $w=(\hl_2^c-\hl_1^c,\ldots,\theta_0-\hl_m^c,\hl_{m+1}^c-\theta_0,\ldots,\hl_{n-1}^c-\hl_{n-2}^c)$,
which together with~\eqref{cond:KKT3} and~\eqref{alpha_i_constr}, yields~\eqref{cond:fenchel2_constr}.
Moreover,~\eqref{cond:KKT3} gives that

\[
\begin{split}
\sum_{i=1}^{m-1}\alpha_i\left( \hl^c_{i+1}-\hl^c_i \right)+\alpha_m&\left( \theta_0-\hl^c_m \right)\\
&+\alpha_{m+1}\left( \hl_{m+1}^c-\theta_0 \right)
+\sum_{m+2}^{n-1}\alpha_i\left( \hl^c_i-\hl^c_{i-1} \right)=0.
\end{split}
\]
Obviously, $\alpha_i=0$ if $\hl^c_i<\hl^c_{i+1}$, for $i=1,\ldots,m-1,m+1,\ldots,n-1$ and~\eqref{def:v_nj0} can be derived as in the proof of Lemma~\ref{lemma:characterization_unconstrained}. 
For the block $B^0_p$ containing $m$, we get that

\[
\begin{split}
\sum_{i\in B_p^0\setminus\{m\}}
\Bigg\{
\frac{\Delta_{(i)}}{v_{np}^0(\beta)}
-
&\left[T_{(i+1)}-T_{(i)}\right]\sum_{l=i+1}^{n}\text{e}^{\beta'Z_{(l)}}
\Bigg\}\\
&+
\frac{\Delta_{(m)}}{v_{np}^0(\beta)}
-
\left[x_0-T_{(m)}\right]\sum_{l=m+1}^{n}\text{e}^{\beta'Z_{(l)}}=0,
\end{split}
\]
which gives exactly~\eqref{def:v_np0}. Therefore showing that the estimator $\hoMLE^0$ defined in~\eqref{def:constrained_estim} satisfies the KKT conditions~\eqref{cond:fenchel1_constr}-\eqref{cond:fenchel2_constr} also proves (ii).

Recall that $\hoMLE^0$ is $\min(\hll_i,\theta_0)$, for $i=1,\ldots,m$, and that $\hll_i$ is the unconstrained estimator when considering only the follow-up times $T_{(1)},\ldots,T_{(m)}$. Moreover, $\hoMLE^0$ is $\max(\hlr_i,\theta_0)$, for $i=m+1,\ldots,n-1$, where $\hlr_i$ can be viewed as the unconstrained estimator when considering only the follow-up times $T_{(m)},\ldots,T_{(n-1)}$. Note that~\eqref{cond:extra} together with~\eqref{cond:fenchel1} imply that

\begin{equation}
\label{cond:extra1}
\sum_{j\leq i}\left\{ \frac{\Delta_{(j)}}{\hl_j}- \left[ T_{(j+1)}-T_{(j)}\right]
\sum_{l=j+1}^{n}
\text{e}^{\beta' Z_{(l)}}\right\}
\geq 0,
\qquad \text{for } i=1,\ldots,n-1.
\end{equation}
The condition holds for $i=1,\ldots,m-1$, and, moreover,

\[
\begin{split}
\sum_{j\leq i}
\Big\{
\frac{\Delta_{(j)}}{\min(\hll_j,\theta_0)}- &\left[ T_{(j+1)}-T_{(j)}\right]
\sum_{l=j+1}^{n}
\text{e}^{\beta' Z_{(l)}}
\Big\}\\
&\geq
\sum_{j\leq i}\left\{ \frac{\Delta_{(j)}}{\hll_j}- \left[ T_{(j+1)}-T_{(j)}\right]
\sum_{l=j+1}^{n}
\text{e}^{\beta' Z_{(l)}}\right\}\geq 0,
\end{split}
\]
for $i=1,\ldots,m-1$. Therefore, $\min(\hll_i,\theta)$, for $i=1,\ldots,m-1$ satisfies~\eqref{cond:fenchel1_constr}.  Furthermore,~\eqref{cond:extra1} holds for $i=m$, which implies that

\[
\begin{split}
&\sum_{j=1}^{m-1}
\left\{
\frac{\Delta_{(j)}}{\min(\hll_j,\theta_0)}- \left[ T_{(j+1)}-T_{(j)}\right]
\sum_{l=j+1}^{n}
\text{e}^{\beta' Z_{(l)}}
\right\}\\
&\quad+
\left\{
\frac{\Delta_{(m)}}{\min(\hll_m,\theta_0)}- \left[ x_0-T_{(m)}\right]
\sum_{l=j+1}^{n}
\text{e}^{\beta' Z_{(l)}}
\right\}\\
&\geq
\sum_{j=1}^m\left\{ \frac{\Delta_{(j)}}{\min(\hll_j,\theta_0)}- \left[ T_{(j+1)}-T_{(j)}\right]
\sum_{l=j+1}^{n}
\text{e}^{\beta' Z_{(l)}}\right\}\geq0,
\end{split}
\]
hence $\hoMLE^0$ satisfies~\eqref{cond:fenchel1_constr_2} as well. It is straightforward that $\max(\hlr_i,\theta_0)$, for $i=m+1,\ldots,n-1$ satisfies~\eqref{cond:fenchel1_constr_1}, since, by definition, $\hlr_i$ satisfies~\eqref{cond:fenchel1}, for $i=m+1,\ldots,n-1$, and

\[
\begin{split}
\sum_{j\geq i}
\Bigg\{
\frac{\Delta_{(j)}}{\max(\hlr_j,\theta_0)}- &\left[ T_{(j+1)}-T_{(j)}\right]
\sum_{l=j+1}^{n}
\text{e}^{\beta' Z_{(l)}}
\Bigg\}\\
&\leq
\sum_{j\geq i}
\left\{
\frac{\Delta_{(j)}}{\hlr_j}- \left[ T_{(j+1)}-T_{(j)}\right]
\sum_{l=j+1}^{n}
\text{e}^{\beta' Z_{(l)}}
\right\}
\leq 0.
\end{split}
\]
Finally, to check if $\hoMLE^0$ verifies the condition~\eqref{cond:fenchel2_constr}, we will argue on the blocks of indices on which $\hoMLE$, and hence $\hll_i$ and $\hlr_i$ are constant. By~\eqref{def:v_nj}, for each block $B_j$, with $j=1,\ldots,k$, on which the unconstrained estimator has the constant value $v_{nj}(\beta)$,

\[
\sum_{i\in B_j}
\left\{
\frac{\Delta_{(i)}}{v_{nj}(\beta)}
-
\left[ T_{(i+1)}-T_{(i)}\right]
\sum_{l=i+1}^{n}
\text{e}^{\beta' Z_{(l)}}
\right\}
v_{nj}(\beta)
=0,
\]
and

\[
\sum_{i\in B_j}
\left\{
\frac{\Delta_{(i)}}{v_{nj}(\beta)}
-
\left[ T_{(i+1)}-T_{(i)}\right]
\sum_{l=i+1}^{n}
\text{e}^{\beta' Z_{(l)}}
\right\}
=0.
\]
Then, on each block $B_j$ that does not contain $m$, we can write

\begin{equation}
\label{block_without_m}
\begin{split}
\sum_{i\in B_j}
\Bigg\{
\frac{\Delta_{(i)}}{\hl_i}
-
&\left[ T_{(i+1)}-T_{(i)}\right]
\sum_{l=i+1}^{n}
\text{e}^{\beta' Z_{(l)}}
\Bigg\}
\hl_i\\
&=
\theta_0
\sum_{i\in B_j}
\left\{
\frac{\Delta_{(i)}}{\hl_i}
-
\left[ T_{(i+1)}-T_{(i)}\right]
\sum_{l=i+1}^{n}
\text{e}^{\beta' Z_{(l)}}
\right\},
\end{split}
\end{equation}
and this holds for $\hll_i$, as well as for $\hlr_i$. It is straightforward that $\min(\hll_i,\theta_0)$, for $i=1,\ldots,m$ and $\max(\hlr_i,\theta_0)$, for $i=m+1,\ldots,n-1$ satisfy this relationship. For the block $B_p$ that contains $m$, we have

\[
\begin{split}
&\sum_{i\in B_p\setminus{\{m\}}}
\left\{
\frac{\Delta_{(i)}}{\hll_i}
-
\left[ T_{(i+1)}-T_{(i)}\right]
\sum_{l=i+1}^{n}
\text{e}^{\beta' Z_{(l)}}
\right\}
\hll_i\\
&\quad\quad+
\Bigg\{
\frac{\Delta_{(m)}}{\hll_m}
-
\left[ T_{(m+1)}-x_0\right]
\sum_{l=m+1}^{n}
\text{e}^{\beta' Z_{(l)}}
-
\left[ x_0-T_{(m)}\right]
\sum_{l=m+1}^{n}
\text{e}^{\beta' Z_{(l)}}
\Bigg\}
\hll_m\\
&\quad=
\theta_0
\sum_{i\in B_p\setminus{\{m\}}}
\left\{
\frac{\Delta_{(i)}}{\hll_i}
-
\left[ T_{(i+1)}-T_{(i)}\right]
\sum_{l=i+1}^{n}
\text{e}^{\beta' Z_{(l)}}
\right\}\\
&\quad\quad+
\theta_0
\left\{
\frac{\Delta_{(m)}}{\hl_m}
-
\left[ T_{(m+1)}-x_0\right]
\sum_{l=m+1}^{n}
\text{e}^{\beta' Z_{(l)}}
-
\left[ x_0-T_{(m)}\right]
\sum_{l=m+1}^{n}
\text{e}^{\beta' Z_{(l)}}
\right\}.
\end{split}
\]
Constraining $\hll_m$ to be $\theta_0$ on the interval $[x_0,T_{(m+1)})$ yields

\begin{equation}
\label{block_with_m}
\begin{split}
\sum_{i\in B_p\setminus{\{m\}}}
&\left\{
\frac{\Delta_{(i)}}{\hll_i}
-
\left[ T_{(i+1)}-T_{(i)}\right]
\sum_{l=i+1}^{n}
\text{e}^{\beta' Z_{(l)}}
\right\}
\hll_i\\
&+
\Bigg\{
\frac{\Delta_{(m)}}{\hll_m}
-
\left[ x_0-T_{(m)}\right]
\sum_{l=m+1}^{n}
\text{e}^{\beta' Z_{(l)}}
\Bigg\}
\hll_m\\
=&
\theta_0
\sum_{i\in B_p\setminus{\{m\}}}
\left\{
\frac{\Delta_{(i)}}{\hll_i}
-
\left[ T_{(i+1)}-T_{(i)}\right]
\sum_{l=i+1}^{n}
\text{e}^{\beta' Z_{(l)}}
\right\}\\
&+
\theta_0
\left\{
\frac{\Delta_{(m)}}{\hll_m}
-
\left[ x_0-T_{(m)}\right]
\sum_{l=m+1}^{n}
\text{e}^{\beta' Z_{(l)}}
\right\}.
\end{split}
\end{equation}
Once more, for $i\in B_p$, $\min(\hll_i,\theta_0)$ satisfies this relationship.
Summing over all blocks in~\eqref{block_without_m} and~\eqref{block_with_m} completes the proof.
\end{proof}
Similar to the unconstrained estimator, we propose $\hoMLE^0(x)=\hoMLE^0(x;\hb)$ as the constrained estimator and $\hat{v}_{nj}^0=v_{nj}^0(\hb)$, where $\hb$ is the maximum partial likelihood estimator.

\begin{remark}
As already pointed out in~\cite{lopuhaa_nane1}, if we take all covariates equal to zero, the characterization of the unconstrained estimator differs slightly from the characterization of the nondecreasing hazard estimator within the ordinary random censorship model, provided in~\cite{huangwellner:1995}. Correspondingly, the characterizations in Lemma~\ref{lemma:characterization_unconstrained} and~\ref{lemma:characterization_constrained}, with all $Z_l\equiv0$ differ from the characterizations furnished in~\cite{banerjee:2008}. Although the estimators in~\cite{banerjee:2008} do not maximize the (pseudo) loglikelihood function in~\eqref{likelihood} (in the absence of covariates and under the null hypothesis) over nondecreasing $\lambda_0$, the asymptotic distribution of the likelihood ratio test based on these estimators will coincide with our proposed distribution, in the case of no covariates.
\end{remark}

Using the notations in~\cite{banerjee:2008}, let $\text{slogcm}(f,I)$ be the left-hand slope of the greatest convex minorant of the restriction of the real-valued function $f$ to the interval $I$. Denote by $\text{slogcm}(f)=\text{slogcm}(f,\mathbbm{R})$. Moreover, let

\[
\text{slogcm}^0(f)=\min\left( \text{slogcm}(f,(-\infty,0]), 0 \right)1_{(-\infty,0]}+\max\left( \text{slogcm}(f,(0,\infty)), 0 \right)1_{(0,\infty)}.
\]
Furthermore, for positive constants $a$ and $b$, define

\begin{equation}
\label{def:X_a,b}
X_{a,b}(t)=a\mathbbm{W}(t)+bt^2,
\end{equation}
where $\mathbbm{W}$ is a standard two-sided Brownian motion originating from zero. Let

\begin{equation}
\label{def:g_a,b}
g_{a,b}(t)=\text{slogcm}(X_{a,b})(t),
\end{equation}
the left-hand slope of the GCM $G_{a,b}$ of the process $X_{a,b}$, at point $t$. The constrained analogous is defined as follows: for $t\leq 0$, construct the GCM of $X_{a,b}$, that will be denoted by $G_{a,b}^L$ and take its left-hand slopes at point $t$, denoted by $D_L(X_{a,b})(t)$. When the slopes exceed zero, replace them by zero. In the same manner, for $t>0$, denote the GCM of $X_{a,b}$ by $G^R_{a,b}$ and its slopes at point $t$ by $D_R(X_{a,b})(t)$. Replace the slopes by zero when they decrease below zero. This slope process will be denoted by $g_{a,b}^0$, which is thus given by

\begin{equation}
\label{def:g_a,b^0}
g_{a,b}^0(t)=
\begin{cases}
\min\left( D_L(X_{a,b})(t),0\right)  & t < 0,\\
0 &  t=0,\\
\max \left( D_R(X_{a,b})(t),0\right) & t>0.
\end{cases}
\end{equation}
Note that for $t\leq 0$, there exists, almost surely $s<0$ such that $D_L(X_{a,b})(s)$ is strictly positive for any point greater than or equal to $s$ and the left derivative at $s$ is non-positive. Equivalently, for $t>0$ there exists almost surely $s>0$ such that $D_R(X_{a,b})(s)$ is strictly negative for any point smaller than or equal to $s$ and the left derivative at $s$ is non-negative.
In addition, observe that $g_{a,b}^0(t)=\text{slogcm}^0(X_{a,b})(t)$, as defined and characterized in~\cite{banerjeewellner:2001}.

\subsection{Nonincreasing baseline hazard}
\label{sec:nonincreasing}
The characterization of the unconstrained and the constrained NPMLE estimators of a nonincreasing baseline hazard function follows analogously to the characterization of the nondecreasing estimators. The unconstrained NPMLE $\hoMLE(x;\beta)$ is obtained by maximizing the (pseudo) likelihood function in~\eqref{likelihood} over all $\lambda_0(T_{(1)})\geq\ldots\geq\lambda(T_{(n)})\geq0$. As derived in~\cite{lopuhaa_nane1}, the likelihood is maximized by a nonincreasing step function that is constant on $(T_{(i-1)},T_{(i)}]$, for $i=1,\ldots,n$ and where $T_{(0)}=0$. The (pseudo) loglikelihood in~\eqref{likelihood} becomes then

\begin{equation}
\label{likelihood_decreasing}
\sum_{i=1}^{n}
\left\{
\Delta_{(i)}\log\lambda_0(T_{(i)})
-
\lambda_0(T_{(i)})
\left[ T_{(i)}-T_{(i-1)}\right]
\sum_{l=i}^n
\text{e}^{\beta' Z_{(l)}}
\right\}.
\end{equation}
The lemmas below provide the characterization of the unconstrained estimator $\hoMLE(x;\beta)$ and the constrained estimator $\hoMLE^0(x;\beta)$. Their proofs follow by arguments similar to those in the proofs of Lemma~\ref{lemma:characterization_unconstrained} and Lemma~\ref{lemma:characterization_constrained}, as well as the necessary and sufficient conditions that uniquely characterize these estimators..

\begin{lemma}
\label{lemma:characterization_unconstrained_dec}
Let $T_{(1)}<\ldots<T_{(n)}$ be the ordered follow-up times and consider a fixed $\beta\in\R^p$.
\begin{enumerate}[(i)]
\item
Let $W_n$ be defined in~\eqref{def:W_n} and
let
\begin{equation}
\label{def:Y_n}
\bar{V}_n(x)=\int\delta\{u\leq x\}\,\mathrm{d}\mathbb{P}_n(u,\delta,z).
\end{equation}
Then, the NPMLE $\hoMLE(x;\beta)$ of a nonincreasing baseline hazard function $\lambda_0$ is given by

\[
\hoMLE(x;\beta)
=
\begin{cases}
\hl_i & T_{(i-1)}< x \leq T_{(i)}, \text{ for }i=1,\ldots,n,\\
0 &  x> T_{(n)},\\
\end{cases}
\]
for $i=1,\ldots,n$, with $T_{(0)}=0$ and where $\hl_i$ is the left derivative of the least concave majorant (LCM) at the point $P_i$
of the cumulative sum diagram consisting of the points
\begin{equation}
\label{def:points_Pj_dec}
P_j=\Big( W_n(\beta,T_{(j)}), \bar{V}_n(T_{(j)}) \Big),
\end{equation}
for $j=1,\ldots,n$ and $P_0=(0,0)$.
\item
Let $B_1,\ldots,B_k$ be blocks of indices such that $\hoMLE(x;\beta)$ is constant on each block and $B_1\cup\ldots\cup B_k=\{1,\ldots,n\}$. Denote by $v_{nj}(\beta)$, the value of the estimator on block $B_j$. Then

\[
v_{nj}(\beta)
=
\frac{\sum_{i\in B_j}\Delta_{(i)}}{\sum_{i\in B_j}\left[ T_{(i)}-T_{(i-1)}\right]
\sum_{l=i}^{n}
\text{e}^{\beta' Z_{(l)}}}.
\]
\end{enumerate}
\end{lemma}
In fact, for $x\geq T_{(n)}$, $\hoMLE(x;\beta)$ can take any value smaller than $\hl_n$, the left derivative of the LCM at the point $P_n$ of the CSD. As before, we propose $\hoMLE(x)=\hoMLE(x;\hb)$ as the estimator of $\lambda_0$ and $\hat{v}_{nj}=v_{nj}(\hb)$,
where $\hb$ denotes the maximum partial likelihood estimator of $\beta_0$. Fenchel conditions as in~\eqref{cond:fenchel1} and~\eqref{cond:fenchel2} can be derived analogously.

The NPMLE estimator $\hoMLE^0$ maximizes the (pseudo) loglikelihood function in~\eqref{likelihood_decreasing} over the set $\lambda_0(T_{(1)})\geq\ldots\geq\lambda_0(T_{(m)})\geq\theta_0\geq\lambda_0(T_{(m+1)})\geq\ldots\geq\lambda_0(T_{(n)})\geq0$. It can be argued that the constrained estimator has to be a nonincreasing step function that is constant on $(T_{(i-1)},T_{(i)}]$, for $i=1,\ldots,n$, is $\theta_0$ on the interval $(T_{(m)},x_0]$, and is zero for $x\geq T_{(n)}$. Hence, the (pseudo) loglikelihood function becomes

\[
\begin{split}
\sum_{i=1}^{m}
&\left\{
\Delta_{(i)}\log\lambda_0(T_{(i)})
-
\lambda_0(T_{(i)})
\left[ T_{(i)}-T_{(i-1)}\right]
\sum_{l=i}^n
\text{e}^{\beta' Z_{(l)}}
\right\}\\
&+\Delta_{(m+1)}\log\lambda_0(T_{(m+1)})
-
\theta_0
\left[ x_0-T_{(m)}\right]
\sum_{l=m+1}^n
\text{e}^{\beta' Z_{(l)}}\\
&-
\lambda_0(T_{(m+1)})
\left[ T_{(m+1)}-x_0\right]
\sum_{l=m+1}^n
\text{e}^{\beta' Z_{(l)}}\\
&+
\sum_{i=m+2}^{n}
\left\{
\Delta_{(i)}\log\lambda_0(T_{(i)})
-
\lambda_0(T_{(i)})
\left[ T_{(i)}-T_{(i-1)}\right]
\sum_{l=i}^n
\text{e}^{\beta' Z_{(l)}}
\right\}.
\end{split}
\]
The characterization of the constrained NPMLE $\hoMLE^0$ is provided with the next lemma.

\begin{lemma}
\label{lemma:characterization_constrained_dec}
Let $x_0\in(0,\tau_H)$ fixed, such that $T_{(m)}< x_0< T_{(m+1)}$, for a given $1\leq m\leq n-1$. Consider a fixed $\beta\in\R^p$.
\begin{enumerate}[(i)]
\item
For $i=1,\ldots,m$, let $\hll_i$ to be the left derivative of the LCM at the point $P_i^L$ of the CSD consisting of the points $P_j^L=P_j$, for $j=1,\ldots,m$, with $P_j$ defined in~\eqref{def:points_Pj_dec}, and $P_0^L=(0,0)$. Moreover, for $i=m+1,\ldots,n$, let $\hlr_i$ be the left derivative of the LCM at the point $P_i^R$ of the CSD consisting of the points $P_j^R=P_j$, for $j=m,m+1\ldots,n$, with $P_j$ defined in~\eqref{def:points_Pj_dec}. Then, the NPMLE $\hoMLE^0(x;\beta)$ of a nonincreasing baseline hazard function $\lambda_0$, under the null hypothesis $H_0:\lambda_0=\theta_0$, is given by

\begin{equation}
\label{def:constrained_estim_dec}
\hoMLE^0(x;\beta)
=
\begin{cases}
\hlo_i & T_{(i-1)}< x \leq T_{(i)}, \text{ for }i=1,\ldots,m,m+2,\ldots,n,\\
\theta_0 & T_{(m)}< x\leq x_0,\\
\hlo_{m+1} & x_0< x \leq T_{(m+1)},\\
0 &  x> T_{(n)},\\
\end{cases}
\end{equation}
where $T_{(0)}=0$ and where $\hlo_i=\max(\hll_i,\theta_0)$, for $i=1,\ldots,m$, and $\hlo_i=\min(\hlr_i,\theta_0)$, for $i=m+1,\ldots,n$.
\item
For $k\geq1$, let $B_1^0,\ldots,B_k^0$ be blocks of indices such that $\hoMLE^0(x;\beta)$ is constant on each block and $B_1^0\cup\ldots\cup B_k^0=\{1,\ldots,n\}$. There is one block, say $B_r^0$, on which $\hoMLE^0(x;\beta)$ is $\theta_0$, and one block, say $B_p^0$, that contains $m+1$. On all other blocks $B_j^0$, denote by $v_{nj}^0(\beta)$ the value of $\hoMLE^0(x;\beta)$ on block $B_j^0$. Then,

\[
v_{nj}^0(\beta)=\frac{\sum_{i\in B^0_j}\Delta_{(i)}}{\sum_{i\in B_j^0}\left[ T_{(j)}-T_{(j-1)}\right]
\sum_{l=j}^{n}
\text{e}^{\beta' Z_{(l)}}}.
\]
On the block $B_p^0$, that contains $m+1$,

\[
\begin{split}
&v_{np}^0(\beta)\\
&=\frac{\sum_{i\in B_p^0}\Delta_{(i)}}{\sum_{i\in B_p^0\setminus\{m+1\}}\left[T_{(i)}-T_{(i-1)}\right]\sum_{l=i+1}^{n}\text{e}^{\beta'Z_{(l)}}+[T_{(m+1)}-x_0]\sum_{l=m+1}^{n}\text{e}^{\beta'Z_{(l)}}}.
\end{split}
\]
\end{enumerate}
\end{lemma}

Evidently, we propose $\hoMLE^0(x)=\hoMLE^0(x;\hb)$ as the constrained estimator of a nonincreasing baseline hazard function $\lambda_0$, as well as $\hat{v}_{nj}^0=v_{nj}^0(\hb)$ on blocks of indices where the estimator is constant. The Fenchel conditions corresponding to~\eqref{cond:fenchel1_constr}-\eqref{cond:fenchel2_constr} can be derived in the same manner as for the constrained estimator in the nondecreasing case.

Let $\text{slolcm}(f,I)$ be the left-hand slope of the LCM of the restriction of the real-valued function $f$ to the interval $I$. Denote by $\text{slolcm}(f)=\text{slolcm}(f,\R)$. For $a,b>0$, let $\bar{X}_{a,b}(t)=a\mathbbm{W}(t)-bt^2$, where $\mathbbm{W}$ is a standard two-sided Brownian motion originating from zero. Denote by $L_{a,b}$ the LCM of $\bar{X}_{a,b}$ and let

\begin{equation}
\label{def:l_a,b}
l_{a,b}(t)=\text{slolcm}(\bar{X}_{a,b})(t),
\end{equation}
be the left-hand slope of $L_{a,b}$, at point $t$. Additionally, set

\[
\text{slolcm}^0(f)=\max\left( \text{slolcm}(f,(-\infty,0]), 0 \right)1_{(-\infty,0]}+\min\left( \text{slolcm}(f,(0,\infty)), 0 \right)1_{(0,\infty)}.
\]
For $t\leq 0$, construct the LCM of $\bar{X}_{a,b}$, that will be denoted by $L_{a,b}^L$ and take its left-hand slope at point $t$, denoted by $D_L(\bar{X}_{a,b})(t)$. When the slopes fall behind zero, replace them by zero. In the same manner, for $t>0$, denote the LCM of $\bar{X}_{a,b}$ by $L^R_{a,b}$ and its slope at point $t$ by $D_R(\bar{X}_{a,b})(t)$. Replace the slopes by zero when they exceed zero. This slope process will be denoted by $l_{a,b}^0$, which is thus given by

\begin{equation}
\label{def:l_a,b^0}
l_{a,b}^0(t)=
\begin{cases}
\max\left( D_L(\bar{X}_{a,b})(t),0\right)  & t < 0,\\
0 &  t=0,\\
\min \left( D_R(\bar{X}_{a,b})(t),0\right) & t>0.
\end{cases}
\end{equation}
Observe that $l_{a,b}^0(t)=\text{slolcm}^0(\bar{X}_{a,b})(t)$.

\section{The limit distribution}
\label{sec:limit_distr}
\noindent
Let $B_{loc}(\mathbbm{R})$ be the space of all locally bounded real functions on $\mathbbm{R}$, equipped with the topology of uniform convergence on compact sets. In addition, $\mathbbm{C}_{min}(\mathbbm{R})$ is defined as the subset of $B_{loc}(\mathbbm{R})$ consisting of continuous functions $f$ for which $f(t)\to\infty$, when $|t|\to\infty$ and $f$ has a unique minimum.
Let $\mathcal{L}$ be the space of locally square integrable real-valued functions on $\mathbbm{R}$, equipped with the topology of $L_2$ convergence on compact sets.\\
For a generic follow-up time $T$, consider $H^{uc}(x)=\mathbbm{P}(T\leq x,\Delta=1)$, the sub-distribution function of the uncensored observations. Moreover, let

\begin{equation}
\label{def:Phi}
\Phi(\beta,x)
=
\int \{u\geq x\}\, \text{e}^{\beta' z}\, \mathrm{d}P(u,\delta,z),
\end{equation}
for $\beta\in\R^p$ and $x\in\R$,
where $P$ is the underlying probability measure corresponding to the distribution of $(T,\Delta,Z)$. For a fixed point $x_0\in (0,\tau_H)$, define the processes

\begin{equation}
\label{processes_neighborhood}
\begin{split}
X_n(x)=n^{1/3}\left( \hoMLE(x_0+n^{-1/3}x) -\theta_0\right),\\
Y_n(x)=n^{1/3}\left( \hoMLE^0(x_0+n^{-1/3}x) -\theta_0\right).
\end{split}
\end{equation}
The following lemma provides the joint asymptotic distribution of the above processes.

\begin{lemma}
\label{lemma:distr_processes}
Assume (A1) and (A2) and let $x_0\in(0,\tau_H)$. Suppose that $\lambda_0$ is nondecreasing on $[0,\infty)$ and continuously differentiable in a neighborhood of $x_0$, with $\lambda_0(x_0)\neq0$ and $\lambda_0'(x_0)>0$. Moreover, assume that the functions $x\mapsto\Phi(\beta_0,x)$ and $H^{uc}(x)$, defined in~\eqref{def:Phi} and above~\eqref{def:Phi}, are continuously differentiable in a neighborhood of $x_0$. Finally, assume that the density of the follow-up times is continuous and bounded away from zero in a neighborhood of $x_0$. Define

\begin{equation}
\label{def:a_b}
a=\sqrt{\frac{\lambda_0(x_0)}{\Phi(\beta_0,x_0)}}\qquad \text{and} \qquad b=\frac{1}{2}\lambda_0'(x_0).
\end{equation}
Then $( X_n,Y_n )$ converges jointly to $( g_{a,b},g_{a,b}^0 )$,
in $\mathcal{L}\times \mathcal{L}$, where the processes $g_{a,b}$ and $g_{a,b}^0$ have been defined in \eqref{def:g_a,b} and \eqref{def:g_a,b^0}.
\end{lemma}

\begin{proof}
Note that the processes $X_n$ and $Y_n$ are monotone. By making use of Corollary 2 in~\cite{huangzhang:1994} and the remark above the corollary, it suffices to prove that the finite dimensional marginals of the process $(X_n,Y_n)$ converge to the finite dimensional marginals of the process $(g_{a,b},g_{a,b}^0)$, in order to prove the lemma.

For $x\geq T_{(1)}$, let

\[
\widehat{W}_n(x)=W_n(\hb,x)-W_n(\hb,T_{(1)}),
\]
where $W_n$ is defined in~\eqref{def:W_n}, and where~$\hb$ is the maximum partial likelihood estimator. For fixed $x_0$ and $x\in[-k,k]$, with $0<k<\infty$, define the process

\begin{equation}
\label{def:Zn}
\begin{split}
\mathbbm{Z}_n(x)
=
\frac{n^{2/3}}{\Phi(\beta_0,x_0)} \Big\{
V_n(x_0&+n^{-1/3}x)-V_n(x_0)\\
&-\lambda_0(x_0)\left[\widehat{W}_n(x_0+n^{-1/3}x)-\widehat{W}_n(x_0)\right]
\Big\},
\end{split}
\end{equation}
where $V_n$ is defined in~\eqref{def:V_n}.
For $a$ and $b$ defined in \eqref{def:a_b}, $\mathbbm{Z}_n$ converges weakly to $X_{a,b}$, as processes in $B_{loc}(\mathbbm{R})$, by Lemma 8 in~\cite{lopuhaa_nane1}.
Define now

\begin{equation}
\label{def:S_n}
S_n(x)=\frac{n^{1/3}}{\Phi(\beta_0,x_0)}\left\{ \widehat{W}_n(x_0+n^{-1/3}x)-\widehat{W}_n(x_0) \right\}.
\end{equation}
From the proof of Lemma 9 in~\cite{lopuhaa_nane1}, $S_n(x)$ converges almost surely to the deterministic function $x$, uniformly on every compact set.

Following the approach in~\cite{groeneboom:1985}, Lopuha\"a and Nane~\cite{lopuhaa_nane1} obtained the asymptotic distribution of the unconstrained maximum likelihood estimator $\hoMLE$ by considering the inverse process

\begin{equation}
\label{def:inverse process}
U_n(z)=\argmin_{x\in[T_{(1)},T_{(n)}]}
\left\{ V_n(x)-z\widehat{W}_n(x)\right\},
\end{equation}
for $z>0$, where the argmin function represents the supremum of times at which the minimum is attained. Since the argmin is invariant under addition of and multiplication with positive constants, it follows that

\[
n^{1/3}
\left[
U_n(\theta_0+n^{-1/3}z)-x_0
\right]
=
\argmin_{x\in I_n(x_0)}\left\{ \mathbbm{Z}_n(x)- S_n(x)z\right\},
\]
where $I_n(x_0)=[-n^{1/3}(x_0-T_{(1)}),n^{1/3}(T_{(n)}-x_0)]$.
For $z>0$, the switching relationship $\hoMLE(x)\leq z$ holds if and only if $U_n(z)\geq x$, with probability one. This translates, in the context of this lemma, to

\[
n^{1/3}
\left[
\hoMLE(x_0+n^{-1/3}x)-\theta_0
\right]\leq z
\Leftrightarrow
n^{1/3}
\left[
U_n(\theta_0+n^{-1/3}z)-x_0
\right] \geq x,
\]
for $0<x_0<\tau_H$ and $\theta_0>0$, with probability one. The switching relationship is thus $X_n(x)\leq z\Leftrightarrow n^{1/3}\left[U_n(\theta_0+n^{-1/3}z)-x_0
\right] \geq x$. Hence finding the limiting distribution of $X_n(x)$ resumes to finding the limiting distribution of $n^{1/3}\left[U_n(\theta_0+n^{-1/3}z)-x_0
\right]$.
By applying Theorem 2.7 in~\cite{kimpollard:1990}, it follows that, for every $z>0$,

\[
n^{1/3}
\left[
U_n(\theta_0+n^{-1/3}z)-x_0
\right]
\law
U(z),
\]
as inferred in the proof of Theorem 2 in~\cite{lopuhaa_nane1}, where 

\[
U(z)=\sup \left\{ t\in\mathbbm{R}: X_{a,b}(t) -zt \text{ is minimal} \right\}.
\]
It will result that, for every $x\in[-k,k]$,

\[
\begin{split}
P\left( X_n(x)\leq z \right)
&=
P\left( n^{1/3}
\left[
\hoMLE(x_0+n^{-1/3}x)-\theta_0
\right]\leq z \right)\\
&=
P\left(
n^{1/3}
\left[
U_n(\theta_0+n^{-1/3}z)-x_0
\right]\geq x
\right)\\
&\to
P\left(
U(z)\geq x
\right).
\end{split}
\]
Using the switching relationship on the limiting process, it can be deduced that $U(z)\geq x \Leftrightarrow g_{a,b}(x)\leq z$,
with probability one, and thus $X_n(x)\law g_{a,b}(x)$.

In order to prove the same type of result for $Y_n(x)$, consider first the following process

\begin{equation}
\label{process_Y_tilde}
\widetilde{Y}_n(x)=n^{1/3}\left( \tilde{\lambda}_n(x_0+n^{-1/3}x)-\theta_0 \right),
\end{equation}
where, for $x_0\in(0,\tau_H)$, such that $T_{(m)}<x_0<T_{(m+1)}$,

\[
\tilde{\lambda}_n(x)
=
\begin{cases}
0 & x < T_{(1)},\\
\hll_i & T_{(i)}\leq x< T_{(i+1)},\, \text{for } i=1,\ldots,m-1\\
\hat{\lambda}_m^L & T_{(m)}\leq x < x_0,\\
0 & x_0\leq x < T_{(m+1)},\\
\hlr_i & T_{(i)}\leq x< T_{(i+1)},\, \text{for } i=m+1,\ldots,n-1\\
\infty & x\geq T_{(n)},
\end{cases}
\]
with $\hll_i$ and $\hlr_i$ defined in Lemma~\ref{lemma:characterization_constrained}. For this, we have considered up to $x_0$ an unconstrained estimator which is constructed based on the sample points $T_{(1)},\ldots,T_{(m+1)}$. Moreover, to the right of $x_0$, we have considered an unconstrained estimator based on the points $T_{(m+1)},\ldots,T_{(n)}$. It is not difficult to see that

\begin{equation}
\label{def:Y_n}
Y_n(x)=
\begin{cases}
\min\left( \widetilde{Y}_n(x),0\right)  & x < 0,\\
0 &  x=0,\\
\max \left( \widetilde{Y}_n(x),0\right) & x>0.
\end{cases}
\end{equation}
For $z>0$, define the inverse processes

\begin{align*}
\label{inverse_proc_constr}
\begin{split}
U_n^L(z)&=
\argmin_{x\in[T_{(1)},T_{(m+1)}]}
\left\{ V_n(x)-z\widehat{W}_n(x)\right\},\\
U_n^R(z)&=
\argmin_{x\in[T_{(m+1)},T_{(n)}]}
\left\{ V_n(x)-z\widehat{W}_n(x)\right\}
\end{split}
\end{align*}
Take $x<x_0$. The switching relationship for $\tilde{\lambda}_n$ is given by $\tilde{\lambda}_n(x)\leq z$ if and only if $U_n^L(z)\geq x$, with probability one, which gives that

\[
n^{1/3}
\left[
\tilde{\lambda}_n(x_0+n^{-1/3}x)-\theta_0
\right]\leq z
\Leftrightarrow
n^{1/3}
\left[
U_n^L(\theta_0+n^{-1/3}z)-x_0
\right] \geq x,
\]
with probability one. Moreover,

\[
n^{1/3}
\left[
U_n^L(\theta_0+n^{-1/3}z)-x_0
\right]
=
\argmin_{x\in I_n^L(x_0)}\left\{ \mathbbm{Z}_n(x)- S_n(x)z\right\},
\]
where $I_n^L(x_0)=[-n^{1/3}(x_0-T_{(1)}),n^{1/3}(T_{(m+1)}-x_0)]$. Denote by

\[
Z_n(z,x)=\mathbbm{Z}_n(x)- S_n(x)z.
\]
As for the unconstrained estimator, we aim to apply Theorem 2.7 in~\cite{kimpollard:1990}. As Theorem 2.7 in~\cite{kimpollard:1990} applies to the argmax of processes on the whole real line, we extend the above process in the following manner

\[
Z_n^{-}(z,x)
=
\begin{cases}
Z_n(z,-n^{1/3}(x_0-T_{(1)})) & x < -n^{1/3}(x_0-T_{(1)}),\\
Z_n(z,x) & -n^{1/3}(x_0-T_{(1)})\leq x\leq n^{1/3}(T_{(m+1)}-x_0),\\
Z_n(z,n^{1/3}(T_{(m+1)}-x_0))+1  & x> n^{1/3}(T_{(m+1)}-x_0).
\end{cases}
\]
Then, $Z_n^{-}(z,x)\in B_{loc}(\mathbbm{R})$ and

\[
n^{1/3}
\left[
U_n^L(\theta_0+n^{-1/3}z)-x_0
\right]
=
\argmin_{x\in\mathbbm{R}}\left\{ Z_n^{-}(z,x) \right\}
=
\argmax_{x\in\mathbbm{R}}\left\{ -Z_n^{-}(z,x) \right\}.
\]
Since $\lambda_0(x_0)=\theta_0>0$ and $\lambda_0$ is continuously differentiable in a neighborhood of $x_0$, it follows by a Taylor expansion and by Lemma 2.5 in~\cite{devroye:1981} that $n^{1/3}(T_{(m+1)}-x_0)=\O_p(n^{-1}\log n)$. Therefore, by virtue of Lemma 8 and Lemma 9 in~\cite{lopuhaa_nane1}, the process $x\mapsto -Z_n^{-}(z,x)$ converges weakly to $Z^{-}(x)\in\mathbbm{C}_{max}(\mathbbm{R})$, for any fixed $z$, where

\[
Z_n^{-}(z,x)
=
\begin{cases}
-X_{a,b}(x)+zx & x\leq 0,\\
1  & x> 0,
\end{cases}
\]
for $a$ and $b$ defined in~\eqref{def:a_b}. Hence, the first condition of Theorem 2.7 in~\cite{kimpollard:1990} is verified. The second condition follows directly from Lemma 11 in~\cite{lopuhaa_nane1}, while the third condition is trivially fulfilled. Thus, for any $z$ fixed,

\[
n^{1/3}
\left[
U_n^L(\theta_0+n^{-1/3}z)-x_0
\right]
\law
U^{-}(z),
\]
where $U^{-}(z)=\sup \left\{ t\leq 0: X_{a,b}(t) -zt \text{ is minimal} \right\}$. Concluding, for $x<0$,

\[
\begin{split}
P\left( \widetilde{Y}_n(x)\leq z \right)
&=
P\left( n^{1/3}
\left[
\tilde{\lambda}_n(x_0+n^{-1/3}x)-\theta_0
\right]\leq z \right)\\
&=
P\left(
n^{1/3}
\left[
U_n^L(\theta_0+n^{-1/3}z)-x_0
\right]\geq x
\right)\\
&\to
P\left(
U^{-}(z)\geq x
\right).
\end{split}
\]
The switching relationship for the limiting process gives that $U^{-}(z)\geq x \Leftrightarrow D_L(X_{a,b})(x)\leq z$, with probability one, where $D_L(X_{a,b})(x)$ has been defined as the left-hand slope of the GCM of $X_{a,b}$, at a point $x<0$. Hence, for $x<0$,

\[
\widetilde{Y}_n(x)\law D_L(X_{a,b})(x).
\]
Completely analogous, $\widetilde{Y}_n(x)\law D_R(X_{a,b})(x)$, for $x>0$. By continuous mapping theorem and by \eqref{def:Y_n}, it can be concluded that for fixed $x\in[-k,k]$,

\[
Y_n(x)\law g_{a,b}^0(x),
\]
where $g^0_{a,b}$ has been defined in \eqref{def:g_a,b^0}.

Our next objective is to apply Theorem 6.1 in~\cite{huangwellner:1995}.
The first condition of Theorem 6.1 is trivially fulfilled. The second condition follows by Lemma 11 in~\cite{lopuhaa_nane1}, while the third condition follows by the definition of the inverse processes. Hence, for fixed $x$,

\[
P\left( X_n(x)\leq z,  Y_n(x)\leq z \right)
\to
P\left(g_{a,b}(x)\leq z, g_{a,b}^0(x)\leq z\right),
\]
for $a$ and $b$ defined in \eqref{def:a_b}. The arguments for one dimensional marginal convergence can be extended to the finite dimensional convergence, as in the proof of Theorem 3.6.2 in~\cite{banerjee_thesis}, by making use of Lemma 3.6.10 in~\cite{banerjee_thesis}. Hence, we can conclude that the finite dimensional marginals of the process $(X_n,Y_n)$ converge to the finite dimensional marginals of the process $(g_{a,b},g_{a,b}^0)$. This completes the proof.

%
\end{proof}

By making use of results in~\cite{lopuhaa_nane1}, a completely similar result holds in the nonincreasing setting.

\begin{lemma}
\label{lemma:distr_processes_dec}
Assume (A1) and (A2) and let $x_0\in(0,\tau_H)$. Suppose that $\lambda_0$ is nonincreasing on $[0,\infty)$ and continuously differentiable in a neighborhood of $x_0$, with $\lambda_0(x_0)\neq0$ and $\lambda_0'(x_0)<0$. Moreover, assume that the functions $x\to\Phi(\beta_0,x)$ and $H^{uc}(x)$, defined in~\eqref{def:Phi} and above~\eqref{def:Phi}, are continuously differentiable in a neighborhood of $x_0$.\\
Then, for $a$ and $b$ defined in~\eqref{def:a_b}, $( X_n,Y_n )$ converge jointly to $\left( l_{a,b},l_{a,b}^0 \right)$
in $\mathcal{L}\times \mathcal{L}$, where the processes $l_{a,b}$ and $l_{a,b}^0$ have been defined in \eqref{def:l_a,b} and \eqref{def:l_a,b^0}.
\end{lemma}

Subsequently, we state two immediate results, that will be used repeatedly throughout the rest of the paper.

\begin{lemma}
\label{relation:D_n}
Let $x_0\in(0,\tau_H)$ fixed and let $\bar{D}_n$ be the set on which the unconstrained NPMLE $\hoMLE$, defined in Lemma~\ref{lemma:characterization_unconstrained}, differs from constrained NPMLE $\hoMLE^0$, defined in Lemma~\ref{lemma:characterization_constrained}. Then, for any $\varepsilon>0$, there exists $k_\varepsilon>0$ such that

\[
\displaystyle{\liminf _{n\to\infty}}\,P\left(\bar{D}_n\subset[x_0-n^{-1/3}k_\varepsilon,x_0+n^{-1/3}k_\varepsilon]\right)\geq 1-\varepsilon.
\]
\end{lemma}

\begin{proof}
The proof of this fact follows by exactly the same reasoning as in the proof of Lemma 2.6 in~\cite{banerjee_preprint:2006}, preprint for~\cite{banerjee:2007}.
\end{proof}

\begin{lemma}
\label{lemma:X_n_Y_n}
Consider the processes $X_n$ and $Y_n$ defined in~\eqref{processes_neighborhood}. Then, for every $\varepsilon>0$ and $k>0$, there exists an $M>0$ such that

\[
\displaystyle{\limsup_{n\to\infty}} \,P \left( \sup_{x\in[-k,k]}\left| X_n(x) \right|>M \right)\leq \varepsilon.
\]
Similarly,

\[
\displaystyle{\limsup_{n\to\infty}} \,P\left( \sup_{x\in[-k,k]}\left| Y_n(x) \right|>M \right)\leq \varepsilon.
\]
\end{lemma}

\begin{proof}
The monotonicity of the processes $X_n$ and $Y_n$ yields that

\[
\begin{split}
\sup_{x\in[-k,k]}\left| X_n(x) \right|=\max\left\{ \left|X_n(-k)\right|,\left|X_n(k)\right| \right\},\\
\sup_{x\in[-k,k]}\left|Y_n(x)\right|=\max\left\{ \left|Y_n(-k)\right|,\left|Y_n(k)\right| \right\}.
\end{split}
\]
Assume $|X_n(k)|$ to be the maximum in the above display. Since for fixed $k$, $X_n(k)\law g_{a,b}(k)$, with $a$ and $b$ defined in \eqref{def:a_b}, it will result that the processes $X_n$ and $Y_n$ in~\eqref{processes_neighborhood} are, with high probability, uniformly bounded.
\end{proof}

The limiting distribution of the likelihood ratio statistic of a nondecreasing baseline hazard function $\lambda_0$ is supplied then by the subsequent theorem.

\begin{theorem}
\label{theorem_inc}
Suppose (A1) and (A2) hold and let $x_0\in(0,\tau_H)$. Assume that $\lambda_0$ is nondecreasing on $[0,\infty)$ and continuously differentiable in a neighborhood of $x_0$, with $\lambda_0(x_0)\neq0$ and $\lambda_0'(x_0)>0$. Moreover, assume that $H^{uc}(x)$ and $x\to\Phi(\beta_0,x)$, defined in~\eqref{def:Phi} and above~\eqref{def:Phi}, are continuously differentiable in a neighborhood of $x_0$. Let $2\log \xi_n(\theta_0)$ be the likelihood ratio statistic for testing $H_0:\lambda_0(x_0)=\theta_0$, as defined in~\eqref{likelihood_ratio}. Then,

\[
2\log \xi_n(\theta_0)\law \mathbbm{D},
\]
where $\mathbbm{D}= \int \left[ (g_{1,1}(u))^2-(g_{1,1}^0(u))^2 \right]\, \mathrm{d}u$, with $g_{1,1}$ and $g_{1,1}^0$ defined in \eqref{def:g_a,b} and \eqref{def:g_a,b^0}.
\end{theorem}

\begin{proof}
By~\eqref{likelihood_inc_reduced} and~\eqref{likelihood_inc_constr_reduced}, the likelihood ratio statistic $2\log \xi_n(\theta_0)=2 L_{\hb}(\hoMLE)-2 L_{\hb}(\hoMLE^0)$ can be expressed as


\[
\begin{split}
2\log \xi_n(\theta_0)
=&2
\sum_{i=1}^{n-1}
\Delta_{(i)}\log\hoMLE(T_{(i)})
-
2\sum_{i=1}^{n-1}
\Delta_{(i)}\log\hoMLE^0(T_{(i)})\\
&-
2\sum_{\substack{i=1\\  i\neq m}}^{n-1}
\left[ T_{(i+1)}-T_{(i)}\right]
\left[
\hoMLE(T_{(i)})
-
\hoMLE^0(T_{(i)})
\right]
\sum_{l=i+1}^n
\text{e}^{\hb' Z_{(l)}}\\
&-2
\left[ T_{(m+1)}-x_0\right]
\left[
\hoMLE(T_{(m)})
-
\theta_0
\right]
\sum_{l=m+1}^n
\text{e}^{\hb' Z_{(l)}}\\
&-2
\left[ x_0-T_{(m)}\right]
\left[
\hoMLE(T_{(m)})
-
\hoMLE^0(T_{(m)})
\right]\sum_{l=m+1}^n
\text{e}^{\hb' Z_{(l)}}.
\end{split}
\]
Let

\begin{equation}
\label{def:T_n}
S_n=2
\sum_{i=1}^{n-1}
\Delta_{(i)}\log\hoMLE(T_{(i)})
-
2\sum_{i=1}^{n-1}
\Delta_{(i)}\log\hoMLE^0(T_{(i)}),
\end{equation}
and denote by $D_n$, the set of indices $i$ on which $\hoMLE(T_{(i)})$ differs from $\hoMLE^0(T_{(i)})$. Hence, expanding both terms of $S_n$ around $\lambda_0(x_0)=\theta_0$, we get

\[
\begin{split}
S_n=&
2
\sum_{i\in D_n}
\Delta_{(i)}\frac{\hoMLE(T_{(i)})-\theta_0}{\theta_0}
-
2\sum_{i\in D_n}
\Delta_{(i)}\frac{\hoMLE^0(T_{(i)})-\theta_0}{\theta_0}\\
&-\sum_{i\in D_n}
\Delta_{(i)}\frac{\left[\hoMLE(T_{(i)})-\theta_0\right]^2}{\theta_0^2}
+
\sum_{i\in D_n}
\Delta_{(i)}\frac{\left[\hoMLE^0(T_{(i)})-\theta_0\right]^2}{\theta_0^2}+R_n,
\end{split}
\]
with

\[
\begin{split}
R_n&=\frac{1}{3}\sum_{i\in D_n}\Delta_{(i)}\frac{\left[\hoMLE(T_{(i)})-\theta_0\right]^3}{\left[\hat{\lambda}_n^*(T_{(i)})\right]^3}
-
\frac{1}{3}\sum_{i\in D_n}\Delta_{(i)}\frac{\left[\hoMLE^0(T_{(i)})-\theta_0\right]^3}{\left[\hat{\lambda}_n^{0*}(T_{(i)})\right]^3}\\
&=R_{n,1}-R_{n,2},
\end{split}
\]
where $\hat{\lambda}_n^*(T_{(i)})$ is a point between $\hoMLE(T_{(i)})$ and $\theta_0$ and $\hat{\lambda}_n^{0*}(T_{(i)})$ is a point between $\hoMLE^0(T_{(i)})$ and $\theta_0$. We want to show that $R_{n,1}$ and $R_{n,2}$, hence $R_n$ converge to zero, in probability. As for the $R_{n,1}$ term, it can be inferred that

\[
|R_{n,1}|\leq \frac{1}{3}\int \delta\{u\in \bar{D}_n\} \frac{\left|n^{1/3}\left( \hoMLE(u)-\theta_0\right)\right|^3}{\left|\hat{\lambda}_n^*(u)\right|^3}\, \mathrm{d}P_n(u,\delta,z),
\]
where $\bar{D}_n$ is the time interval on which $\hoMLE$ differs from $\hoMLE^0$. Choose now $\varepsilon>0$ and $\gamma>0$, and for $x_0\in(0,\tau_H)$ fixed and $k_{\varepsilon}>0$, denote by $I_n=[x_0-n^{-1/3}k_\varepsilon,x_0+n^{-1/3}k_\varepsilon]$. We can write $R_{n,1}=R_{n,1}\{\bar{D}_n\subset I_n\}+R_{n,1}\{\bar{D}_n\nsubset I_n\}$. Since, by Lemma~\ref{relation:D_n},

\[\mathbbm{P}(|R_{n,1}\{\bar{D}_n\nsubset I_n\}|>\gamma)\leq \mathbbm{P}(\bar{D}_n\nsubset I_n)<\varepsilon,
\]
we will further focus on bounding $|R_{n,1}\{\bar{D}_n\subset I_n\}|$.
By Lemma~\ref{relation:D_n} and by Lemma~\ref{lemma:X_n_Y_n}, there exists $k_\varepsilon>0$ such that $\sup_{x\in[-k_\varepsilon,k_\varepsilon]}\left| \hoMLE(x_0+n^{-1/3}x)-\theta_0\right|$ is $\mathcal{O}_p(n^{-1/3})$. Furthermore, since

\[
\sup_{x\in[-k_\varepsilon,k_\varepsilon]}\left| \hoMLE^*(x_0+n^{-1/3}x)-\theta_0\right|\leq \sup_{x\in[-k_\varepsilon,k_\varepsilon]}\left| \hoMLE(x_0+n^{-1/3}x)-\theta_0\right|,
\]
it will result that, for $u\in\bar{D}_n$, $\left|n^{1/3}\left( \hoMLE(u)-\theta_0\right)\right|^3$ is uniformly bounded and $\left|\hat{\lambda}_n^*(u)\right|^3$ is uniformly bounded away from zero. It will result that there exists $M>0$ such that

\[
\begin{split}
|R_{n,1}|\leq & M \int \delta \{x_0-k_{\varepsilon}n^{-1/3}\leq u\leq x_0+k_{\varepsilon}n^{-1/3}\}\mathrm{d}\left(P_n-P\right)(u,\delta,z)\\
&+
M \int \delta \{x_0-k_{\varepsilon}n^{-1/3}\leq u\leq x_0+k_{\varepsilon}n^{-1/3}\}\mathrm{d}P(u,\delta,z)
+
o_p(1).
\end{split}
\]
Chebyshev's inequality provides that the first term on the right-hand side is $\O_p(n^{-2/3})$. As the function $H^{uc}$ defined above~\eqref{def:Phi} is assumed to be continuously differentiable in a neighborhood of $x_0$, the second term on the right-hand side is $\O_p(n^{-1/3})$. We can conclude that $R_{n,1}=o_p(1)$. Completely similar, by using Lemma~\ref{relation:D_n} and Lemma~\ref{lemma:X_n_Y_n}, it can be shown that $R_{n,2}=o_p(1)$. Thus $2\log \xi_n(\theta_0)=A_n-B_n+o_p(1)$, where

\begin{equation}
\label{def:A_n}
\begin{split}
A_n
=&
\frac{2}{\theta_0}
\sum_{i\in D_n}
\Delta_{(i)}
\left[\hoMLE(T_{(i)})
-
\hoMLE^0(T_{(i)})
\right]\\
&-2\sum_{i\in D_n\setminus{\{m\}}}
\left[ T_{(i+1)}-T_{(i)}\right]
\left[
\hoMLE(T_{(i)})
-
\hoMLE^0(T_{(i)})
\right]
\sum_{l=i+1}^n
\text{e}^{\hb' Z_{(l)}}\\
&-2\left[ T_{(m+1)}-x_0\right]
\left[
\hoMLE(T_{(m)})
-
\theta_0
\right]
\sum_{l=m+1}^n
\text{e}^{\hb' Z_{(l)}}\\
&-2\left[ x_0-T_{(m)}\right]
\left[
\hoMLE(T_{(m)})
-
\hoMLE^0(T_{(m)})
\right]
\sum_{l=m+1}^n
\text{e}^{\hb' Z_{(l)}},
\end{split}
\end{equation}
and

\begin{equation}
\label{def:B_n}
B_n=\frac{1}{\theta_0^2}\sum_{i\in D_n}
\Delta_{(i)} \left\{
\left[\hoMLE(T_{(i)})
-
\theta_0
\right]^2-\left[\hoMLE^0(T_{(i)})
-
\theta_0
\right]^2
\right\}.
\end{equation}
Hence, $A_n$ can be written as $A_n=A_{n1}-A_{n2}$, where

\[
A_{n1}=\frac{2}{\theta_0}
\sum_{i\in D_n}
\left[\hoMLE(T_{(i)})
-
\theta_0
\right]
\left\{
\Delta_{(i)}-\theta_0\left[ T_{(i+1)}-T_{(i)}\right]\sum_{l=i+1}^n
\text{e}^{\hb' Z_{(l)}}
\right\},
\]
and

\[
\begin{split}
A_{n2}=&\frac{2}{\theta_0}\sum_{i\in D_n\setminus{\{m\}}}
\left[\hoMLE^0(T_{(i)})
-
\theta_0
\right]
\left\{
\Delta_{(i)}-\theta_0\left[ T_{(i+1)}-T_{(i)}\right]\sum_{l=i+1}^n
\text{e}^{\hb' Z_{(l)}}
\right\}\\
&+
\frac{2}{\theta_0}
\left[\hoMLE^0(T_{(m)})
-
\theta_0
\right]
\left\{
\Delta_{(m)}-\theta_0\left[ x_0-T_{(m)}\right]\sum_{l=m+1}^n
\text{e}^{\hb' Z_{(l)}}
\right\}.
\end{split}
\]
For the term $A_{n1}$, partition the set of indices $D_n$ into $s$ consecutive blocks of indices $B_1,\ldots,B_s$, such that $\hoMLE$ is constant on each block. Denote by $\hat{v}_{nj}$ the unconstrained estimator $\hoMLE(T_{(i)})$, for each $i\in B_j$, with $j=1,\ldots,s$. By~\eqref{def:v_nj}, it follows that

\[
\begin{split}
A_{n1}&=\frac{2}{\theta_0}\sum_{j=1}^s\sum_{i\in B_j}\left( \hat{v}_{nj}-\theta_0 \right)\left\{ \Delta_{(i)}-\theta_0 \left[ T_{(i+1)}-T_{(i)}\right]\sum_{l=i+1}^n
\text{e}^{\hb' Z_{(l)}}\right\}\\
&=\frac{2}{\theta_0}\sum_{j=1}^s\left( \hat{v}_{nj}-\theta_0 \right)\left\{ \sum_{i\in B_j}\Delta_{(i)}-\theta_0 \sum_{i\in B_j}\left[ T_{(i+1)}-T_{(i)}\right]\sum_{l=i+1}^n
\text{e}^{\hb' Z_{(l)}}\right\}\\
&=\frac{2}{\theta_0}\sum_{j=1}^s\left( \hat{v}_{nj}-\theta_0 \right)^2 \sum_{i\in B_j}\left[ T_{(i+1)}-T_{(i)}\right]\sum_{l=i+1}^n
\text{e}^{\hb' Z_{(l)}}\\
&=\frac{2}{\theta_0}n \sum_{i\in D_n}\left[ \hoMLE(T_{(i)})-\theta_0 \right]^2\frac{1}{n} \left[ T_{(i+1)}-T_{(i)}\right]\sum_{l=i+1}^n
\text{e}^{\hb' Z_{(l)}}.
\end{split}
\]
Define

\begin{equation}
\label{def:Phi_n}
\Phi_n(\beta,x)=\int \{u\geq x\}\, \text{e}^{\beta' z}\, \mathrm{d}P_n(u,\delta,z),
\end{equation}
and note that

\[
\int_{[T_{(i)},T_{(i+1)})}\Phi_n(\hb,u)\,\mathrm{d}u
=
\frac{1}{n}\left[T_{(i+1)}-T_{(i)} \right]
\sum_{l=i+1}^n
\text{e}^{\hb' Z_{(l)}},
\]
for each $i=1,\ldots,n-1$.
The term $A_{n1}$ can then be written as

\[
A_{n1}
=
\frac{2}{\theta_0}n
\int\left\{u\in\bar{D}_n\right\} \left[ \hoMLE(u)-\theta_0 \right]^2 \Phi_n(\hb,u)\, \mathrm{d}u,
\]
where $\bar{D}_n$ is the interval on which $\hoMLE$ and $\hoMLE^0$ differ.
Similarly, for the term $A_{n2}$, partition $D_n$ into $q$ consecutive blocks of indices $B_1^0,\ldots,B_q^0$, such that the constrained estimator $\hoMLE^0$ is constant on each block. There is one block, say $B^0_r$, on which the constrained estimator is $\theta_0$, and one block, say $B_p^0$ that contains $m$. On all other blocks $B^0_j$, denote by $\hat{v}_{nj}^0$ the constrained estimator $\hoMLE^0(T_{(i)})$, for each $i\in B^0_j$. It will result that,

\[
\begin{split}
A_{n2}=&\frac{2}{\theta_0}\sum_{\substack{j=1\\j\neq r,p}}^q\sum_{i\in B_j^0}\left( \hat{v}_{nj}^0-\theta_0 \right)\left\{ \Delta_{(i)}-\theta_0 \left[ T_{(i+1)}-T_{(i)}\right]\sum_{l=i+1}^n
\text{e}^{\hb' Z_{(l)}}\right\}\\
&+
\frac{2}{\theta_0}\sum_{i\in B_p^0\setminus{\{m\}}}\left( \hat{v}_{np}^0-\theta_0 \right)\left\{ \Delta_{(i)}-\theta_0 \left[ T_{(i+1)}-T_{(i)}\right]\sum_{l=i+1}^n
\text{e}^{\hb' Z_{(l)}}\right\}\\
&+
\frac{2}{\theta_0}\left( \hat{v}_{np}^0-\theta_0 \right)\left\{ \Delta_{(m)}-\theta_0 \left[ x_0-T_{(m)}\right]\sum_{l=m+1}^n
\text{e}^{\hb' Z_{(l)}}\right\}\\
=&
\frac{2}{\theta_0}\sum_{\substack{j=1\\j\neq r,p}}^q\left( \hat{v}_{nj}^0-\theta_0 \right)\left\{ \sum_{i\in B_j^0}\Delta_{(i)}-\theta_0 \sum_{i\in B_j^0}\left[ T_{(i+1)}-T_{(i)}\right]\sum_{l=i+1}^n
\text{e}^{\hb' Z_{(l)}}\right\}\\
&+
\frac{2}{\theta_0}\left( \hat{v}_{np}^0-\theta_0 \right)
\Bigg\{
\sum_{i\in B_p^0}\Delta_{(i)}-\theta_0
\bigg[
\sum_{i\in B_p^0\setminus{\{m\}}}\left[ T_{(i+1)}-T_{(i)}\right]\sum_{l=i+1}^n
\text{e}^{\hb' Z_{(l)}}\\
&\qquad \qquad\qquad\qquad\qquad
\qquad\qquad\quad+
\left[ x_0-T_{(m)}\right]\sum_{l=m+1}^n
\text{e}^{\hb' Z_{(l)}}
\bigg]
\Bigg\}.
\end{split}
\]
By~\eqref{def:v_nj0} and~\eqref{def:v_np0},

\[
\begin{split}
A_{n2}=&\frac{2}{\theta_0}\sum_{\substack{j=1\\j\neq r,p}}^q\left( \hat{v}_{nj}^0-\theta_0 \right)^2 \sum_{i\in B_j^0}\left[ T_{(i+1)}-T_{(i)}\right]\sum_{l=i+1}^n
\text{e}^{\hb' Z_{(l)}}\\
&+
\frac{2}{\theta_0}\left( \hat{v}_{np}^0-\theta_0 \right)^2
\Bigg\{
\sum_{i\in B_p^0\setminus{\{m\}}}\left[ T_{(i+1)}-T_{(i)}\right]\sum_{l=i+1}^n
\text{e}^{\hb' Z_{(l)}}\\
&\qquad \qquad\qquad\qquad\qquad+
\left[ x_0-T_{(m)}\right]\sum_{l=m+1}^n
\text{e}^{\hb' Z_{(l)}}
\Bigg\}\\
=&\frac{2}{\theta_0}n \sum_{i\in D_n\setminus{\{m\}}}\left[ \hoMLE^0(T_{(i)})-\theta_0 \right]^2\frac{1}{n} \left[ T_{(i+1)}-T_{(i)}\right]\sum_{l=i+1}^n
\text{e}^{\hb' Z_{(l)}}\\
&+
\frac{2}{\theta_0}n\left[ \hoMLE^0(T_{(m)})-\theta_0 \right]^2\frac{1}{n} \left[ x_0-T_{(m)}\right]\sum_{l=m+1}^n
\text{e}^{\hb' Z_{(l)}}.
\end{split}
\]
As $\hoMLE^0(x)=\hoMLE^0(T_{(m)})$ on the interval $[T_{(m)},x_0)$ and $\hoMLE^0(x)=\theta_0$ on the interval $[x_0,T_{(m+1)})$, this gives that

\[
\begin{split}
\int_{T_{(m)}}^{T_{(m+1)}}&\left[ \hoMLE^0(u)-\theta_0 \right]^2\Phi_n(\hb,u)\,\mathrm{d}u\\
=&
\int_{T_{(m)}}^{x_0}\left[ \hoMLE^0(u)-\theta_0 \right]^2\Phi_n(\hb,u)\,\mathrm{d}u
+
\int_{x_0}^{T_{(m+1)}}\left[ \hoMLE^0(u)-\theta_0 \right]^2\Phi_n(\hb,u)\,\mathrm{d}u\\
=&
\frac{1}{n}
\left[ \hoMLE^0(T_{(m)})-\theta_0 \right]^2
\left[ x_0-T_{(m)}\right]\sum_{l=m+1}^n
\text{e}^{\hb' Z_{(l)}}.
\end{split}
\]
This leads to

\[
A_{n2}
=
\frac{2}{\theta_0}n
\int \left\{u\in\bar{D}_n\right\} \left[ \hoMLE^0(u)-\theta_0 \right]^2 \Phi_n(\hb,u)\, \mathrm{d}u,
\]
and, thus $A_n$ in~\eqref{def:A_n} can be written as

\[
A_n=\frac{2}{\theta_0}n \int \left\{u\in\bar{D}_n\right\}
\left\{\left[ \hoMLE(u)-\theta_0 \right]^2- \left[ \hoMLE^0(u)-\theta_0 \right]^2 \right\}\Phi_n(\hb,u)\, \mathrm{d}u.
\]
In a similar manner, $B_n$ in~\eqref{def:B_n} can be expressed as

\[
B_n=\frac{1}{\theta_0^2}n\int \left\{u\in\bar{D}_n\right\}
\left\{\left[ \hoMLE(u)-\theta_0 \right]^2- \left[ \hoMLE^0(u)-\theta_0 \right]^2 \right\}\, \mathrm{d}V_n(u),
\]
by \eqref{def:V_n} and by noting that for every $i=1,\ldots,n-1$,

\[
\int_{[T_{(i)},T_{(i+1)})}\,\mathrm{d}V_n(u)
=
V_n(T_{(i+1)})-V_n(T_{(i)})
=
\frac{1}{n}\Delta_{(i)}.
\]
Concluding,

\[
\begin{split}
2\log \xi_n(\theta_0)=&\frac{2}{\theta_0}n \int \left\{u\in\bar{D}_n\right\}
\left\{\left[ \hoMLE(u)-\theta_0 \right]^2- \left[ \hoMLE^0(u)-\theta_0 \right]^2 \right\}\Phi_n(\hb,u)\, \mathrm{d}u\\
&-\frac{1}{\theta_0^2}n\int
\left\{u\in\bar{D}_n\right\}
 \left\{\left[ \hoMLE(u)-\theta_0 \right]^2- \left[ \hoMLE^0(u)-\theta_0 \right]^2 \right\}\, \mathrm{d}V_n(u)+o_p(1).
\end{split}
\]
Let $V(x)=\int \delta \{u<x\}\,\mathrm{d}P(u,\delta,z)$, and see that, in fact, $V(x)=H^{uc}(x)$, where $H^{uc}$ has been defined above~\eqref{def:Phi}. Thus,

\[
\begin{split}
2\log \xi_n(\theta_0)=&\frac{2}{\theta_0}n \int
\left\{u\in\bar{D}_n\right\}
\left\{\left[ \hoMLE(u)-\theta_0 \right]^2- \left[ \hoMLE^0(u)-\theta_0 \right]^2 \right\}\Phi(\beta_0,u)\, \mathrm{d} u\\
&-\frac{1}{\theta_0^2}n\int
\left\{u\in\bar{D}_n\right\}
 \left\{\left[ \hoMLE(u)-\theta_0 \right]^2- \left[ \hoMLE^0(u)-\theta_0 \right]^2 \right\}\, \mathrm{d}V(u)+\bar{R}_n+o_p(1),
\end{split}
\]
where $\bar{R}_n=\bar{R}_{n1}-\bar{R}_{n2}$, with

\[
\begin{split}
\bar{R}_{n1}=\frac{2}{\theta_0}n \int
\left\{u\in\bar{D}_n\right\}
\Bigg\{
&\left[ \hoMLE(u)-\theta_0 \right]^2\\
&- \left[ \hoMLE^0(u)-\theta_0 \right]^2 
\Bigg\}
\left( \Phi_n(\hb,u)-\Phi(\beta_0,u)\right)\, \mathrm{d}u,
\end{split}
\]
and

\[
\bar{R}_{n2}=\frac{1}{\theta_0^2}n \int
\left\{u\in\bar{D}_n\right\}
\left\{\left[ \hoMLE(u)-\theta_0 \right]^2- \left[ \hoMLE^0(u)-\theta_0 \right]^2 \right\}\, \mathrm{d}\left(V_n(u)-V(u)\right).
\]
The aim is to show that $\bar{R}_{n1}$ and $\bar{R}_{n2}$, and thus $\bar{R}_n$ is $o_p(1)$. The term $\bar{R}_{n1}$ can be written as

\[
\begin{split}
\frac{2}{\theta_0}n^{1/3} \int
\left\{u\in\bar{D}_n\right\}
\Bigg\{
&\left[ n^{1/3}\left(\hoMLE(u)-\theta_0\right) \right]^2\\
&- \left[ n^{1/3}\left(\hoMLE^0(u)-\theta_0 \right)\right]^2 
\Bigg\}
\left(\Phi_n(\hb,u)-\Phi(\beta_0,u)\right)\, \mathrm{d}u.
\end{split}
\]
Lemma 4 in~\cite{lopuhaa_nane1} provides that

\[
\sup_{x \in \mathbb{R}}
\left|\Phi_n(\hb,x)-\Phi(\beta_0,x)\right|\to 0,
\]
with probability one. From Lemma~\ref{lemma:X_n_Y_n} and since $\int\{u\in\bar{D}_n\}\,\mathrm{d}u\leq 2k_{\varepsilon}n^{-1/3}$, by Lemma~\ref{relation:D_n} and by using similar arguments as for the term $R_{n,1}$, we can conclude that $\bar{R}_{n1}$ is $o_p(1)$. Analogously,

\[
\begin{split}
\bar{R}_{n2}=\frac{1}{\theta_0^2}n^{1/3} \int
\left\{u\in\bar{D}_n\right\}
\Bigg\{
&\left[ n^{1/3}\left(\hoMLE(u)-\theta_0\right) \right]^2\\
&- \left[ n^{1/3}\left(\hoMLE^0(u)-\theta_0 \right)\right]^2 
\Bigg\}
\delta \, \mathrm{d}(P_n-P)(u,\delta,z).
\end{split}
\]
Once more, by Lemma~\ref{relation:D_n} and Lemma~\ref{lemma:X_n_Y_n}, there exists $M_2>0$ such that

\[
|\bar{R}_{n2}|\leq \frac{M_2^2}{\theta_0^2}n^{1/3}\int\delta \left\{u\in\bar{D}_n\right\}\, \mathrm{d}(P_n-P)(u,\delta,z),
\]
with arbitrarily large probability.
%
Chebyshev's inequality along with the same reasoning as for the term $R_{n,1}$ provides that $\bar{R}_{n2}=o_p(1)$.
Hence,

\[
\begin{split}
2\log \xi_n(\theta_0)=&\frac{2}{\theta_0}n \int
\left\{u\in\bar{D}_n\right\}
\left\{\left[ \hoMLE(u)-\theta_0 \right]^2- \left[ \hoMLE^0(u)-\theta_0 \right]^2 \right\}\Phi(\beta_0,u) \,\mathrm{d}(u)\\
&-\frac{1}{\theta_0^2}n\int
\left\{u\in\bar{D}_n\right\}
\left\{\left[ \hoMLE(u)-\theta_0 \right]^2- \left[ \hoMLE^0(u)-\theta_0 \right]^2 \right\}\, \mathrm{d}V(u)+o_p(1).
\end{split}
\]
Consider the change of variable $x=n^{1/3}(u-x_0)$ and let $\widetilde{D}_n=n^{1/3}\left( \bar{D}_n-x_0 \right)$. This yields that

\[
\begin{split}
2\log \xi_n(\theta_0)
=&\frac{2}{\theta_0} \int
\left\{x\in\widetilde{D}_n\right\}
\left[X_n^2(x)- Y_n^2(x) \right] \Phi(\beta_0,x_0+n^{-1/3}x)\, \mathrm{d}x\\
&-\frac{1}{\theta_0^2}\int
 \left\{x\in\widetilde{D}_n\right\}
\left[ X_n^2(x)- Y_n^2(x)^2 \right] V'(x_0+n^{-1/3}x)\,\mathrm{d}x+o_p(1)\\
=&\frac{2}{\theta_0} \Phi(\beta_0,x_0)\int \left\{x\in\widetilde{D}_n\right\} \left[X_n^2(x)- Y_n^2(x) \right] \, \mathrm{d}x\\
&-\frac{1}{\theta_0^2}V'(x_0)\int \left\{x\in\widetilde{D}_n\right\} \left[ X_n^2(x)- Y_n^2(x) \right] \,\mathrm{d}x+o_p(1).
\end{split}
\]
As inferred in~\cite{lopuhaa_nane1},

\[
\lambda_0(x)=\frac{\mathrm{d}V(x)/\mathrm{d}x}{\Phi(\beta_0,x)},
\]
which gives that

\[
2\log \xi_n(\theta_0)=\frac{1}{\theta_0} \Phi(\beta_0,x_0)\int \left\{x\in\widetilde{D}_n\right\} \left[X_n^2(x)- Y_n^2(x) \right] \, \mathrm{d}x+o_p(1).
\]
Thus

\[
2\log \xi_n(\theta_0)=\frac{1}{a^2} \int \left\{x\in\widetilde{D}_n\right\} \left[X_n^2(x)- Y_n^2(x) \right] \, \mathrm{d}x+o_p(1),
\]
where $a$ has been defined in~\eqref{def:a_b}. From Lemma~\ref{relation:D_n}, for every $\varepsilon>0$, we can find an interval $[-k_\varepsilon,k_\varepsilon]$ such that $\mathbbm{P}(\widetilde{D}_n\subset[-k_\varepsilon,k_\varepsilon])>1-\varepsilon$, for $n$ sufficiently large. In order to prove the theorem, we apply Lemma 4.2 in~\cite{prakasarao:1969}, by taking

\[
\begin{split}
Q_n=&\frac{1}{a^2}\int \left\{x\in\widetilde{D}_n\right\} \left[X_n^2(x)- Y_n^2(x) \right] \, \mathrm{d}x,\\
Q_{n\varepsilon}=&\frac{1}{a^2}\int \left\{x\in[-k_\varepsilon,k_\varepsilon]\right\} \left[X_n^2(x)- Y_n^2(x) \right] \, \mathrm{d}x,\\
Q_{\varepsilon}=&\frac{1}{a^2}\int \left\{x\in[-k_\varepsilon,k_\varepsilon]\right\} \left[(g_{a,b}(x))^2- \left(g_{a,b}^0(x)\right)^2 \right] \, \mathrm{d}x,
\end{split}
\]
and
\[
Q=\frac{1}{a^2}\int \left\{x\in D_{a,b}\right\} \left[(g_{a,b}(x))^2- \left(g_{a,b}^0(x)\right)^2 \right] \, \mathrm{d}x,
\]
where $D_{a,b}$ denotes the set on which $g_{a,b}$ and $g_{a,b}^0$ differ. Condition (i) in Lemma 4.2 of Prakasa Rao follows by Lemma \ref{relation:D_n}. In addition, Lemma \ref{relation:D_n} and Lemma \ref{lemma:distr_processes} yield condition (ii), since for every $\varepsilon>0$, we can find $k_\varepsilon>0$ such that $\mathbbm{P}(D_{a,b}\subset[-k_\varepsilon,k_\varepsilon])>1-\varepsilon$. The third condition follows, for every fixed $\varepsilon$, by Lemma~\ref{lemma:distr_processes} and by continuous mapping theorem. Namely, $(X_n,Y_n)\Rightarrow (g_{a,b},g_{a,b}^0)$ as a process in $\mathcal{L}\times\mathcal{L}$ and $(f,g)\mapsto\int\{x\in[-c,c]\}(f^2(x)-g^2(x))\, \mathrm{d}x$ is a continuous function defined on $\mathcal{L}\times\mathcal{L}$ with values in $\mathbbm{R}$. Conclusively,

\[
\begin{split}
\frac{1}{a^2}\int \left[X_n^2(x)- Y_n^2(x) \right] &\left\{x\in\widetilde{D}_n\right\}\, \mathrm{d}x\\
&\law
\frac{1}{a^2} \int \left[(g_{a,b}(x))^2- \left(g_{a,b}^0(x)\right)^2 \right] \left\{x\in D_{a,b}\right\}\, \mathrm{d}x,\\
&\eqd \int \left[(g_{1,1}(x))^2- \left(g_{1,1}^0(x)\right)^2 \right] \left\{x\in D_{1,1}\right\}\, \mathrm{d}x,
\end{split}
\]
by continuous mapping theorem and by Brownian scaling, as derived in~\cite{banerjeewellner:2001}. This completes the proof.
\end{proof}

The asymptotic distribution of the likelihood ratio statistic in the nonincreasing baseline hazard setting can be derived completely analogous.

\begin{theorem}
Suppose (A1) and (A2) hold and let $x_0\in(0,\tau_H)$. Assume that $\lambda_0$ is nonincresing on $[0,\infty)$ and continuously differentiable in a neighborhood of $x_0$, with $\lambda_0(x_0)\neq0$ and $\lambda_0'(x_0)<0$. Moreover, assume that $H^{uc}(x)$ and $x\to\Phi(\beta_0,x)$, defined in~\eqref{def:Phi} and above~\eqref{def:Phi}, are continuously differentiable in a neighborhood of $x_0$. Let $2\log \xi_n(\theta_0)$ be the likelihood ratio statistic for testing $H_0:\lambda_0(x_0)=\theta_0$, as defined in~\eqref{likelihood_ratio}. Then,

\[
2\log \xi_n(\theta_0)\law \mathbbm{D}.
\]
\end{theorem}

\begin{proof}
Following the same reasoning as in the proof of Theorem~\ref{theorem_inc} and by Lemma~\ref{lemma:distr_processes_dec}, it can be deduced that

\[
2\log \xi_n(\theta_0)
\law
\frac{1}{a^2}\int \left[(l_{a,b}(x))^2- \left(l_{a,b}^0(x)\right)^2 \right] \left\{x\in \bar{D}_{a,b}\right\}\, \mathrm{d}x,
\]
where $\bar{D}_{a,b}$ is the set on which $l_{a,b}$ and $l_{a,b}^0$ differ. By continuous mapping theorem, it suffices to show that, for $t$ fixed, $l_{a,b}(\bar{X}_{a,b})(t)$ has the same distribution as $g_{a,b}(X_{a,b})(t)$ and $l_{a,b}^{0}(\bar{X}_{a,b})(t)$ has the same distribution as $g_{a,b}^{0}(X_{a,b})(t)$. It is noteworthy that

\[
\text{slolcm}(\bar{X}_{a,b})(t)=-\text{slogcm}(-\bar{X}_{a,b})(t).
\]
Thus, by Brownian motion properties and continuous mapping theorem,

\[
\begin{split}
P\left(
l_{a,b}(t)\leq z
\right)
&=
P\left(
-\text{slogmc}
(-a\mathbbm{W}(t)+t^2)
\leq z
\right)\\
&=
P\left(
-\text{slogmc}
(a\mathbbm{W}(t)+t^2)
\leq z
\right)
=
P\left(
-g_{a,b}(t)\leq z
\right).
\end{split}
\]
Concluding, $l_{a,b}(\bar{X}_{a,b})(t)\eqd -g_{a,b}(X_{a,b})(t)$, and a similar reasoning can be applied to show that $l_{a,b}^0(\bar{X}_{a,b})(t)\eqd -g_{a,b}^0(X_{a,b})(t)$. The proof is then immediate, by continuous mapping theorem.
\end{proof}

\begin{remark}
The same limiting distribution $\mathbbm{D}$ is obtained for the loglikleihood ratio statistic in the absence of covariates in~\cite{banerjee:2008}, as well as in other censoring frameworks, as derived in~\cite{banerjeewellner:2001}. In fact, it has been shown in~\cite{banerjee:2007} that the same holds true for a wide class of monotone response models. This distribution differs from the usual $\chi_1^2$ distribution, that is obtained in the regular parametric setting. It is noteworthy that $\mathbbm{D}$ does not depend on any of the parameters of the underlying model, and this property turns out to be particularly useful in constructing confidence intervals for the parameters of interest, as it will be exposed in the subsequent section.
\end{remark}

\section{Pointwise confidence intervals via simulations}
\label{sec:simulations}
Once having derived the asymptotic distribution of the likelihood ratio statistic, the practical application at hand is to construct, for fixed $x_0\in(0,\tau_H)$, pointwise confidence intervals. We will derive such intervals, for a nondecreasing baseline hazard function $\lambda_0$, evaluated at a fixed point $x_0$, based on simulated data and compare these intervals with the intervals based on the asymptotic distribution of the nondecreasing NPMLE $\hoMLE$. According to Theorem 2 in~\cite{lopuhaa_nane1}, for fixed $x_0$,

\[
\begin{split}
n^{1/3}\left (
\hoMLE(x_0)-\lambda_0(x_0)
\right)
\law&
\left(
\frac{4 \lambda_0(x_0)\lambda_0'(x_0)}{\Phi(\beta_0,x_0)}
\right)^{1/3}
\argmin_{x\in\mathbbm{R}}\{\mathbbm{W}(t)+t^2\}\\
&\equiv
C(x_0) \mathbbm{Z},
\end{split}
\]
where $\mathbbm{W}$ is standard two-sided Brownian motion starting from zero, and the constant 
$C(x_0)$ depends on $x_0$ and on the underlying parameters. An estimator $\hat{C}_n(x_0)$ of $C(x_0)$ will then yield an $1-\alpha$ confidence interval for $\lambda_0(x_0)$

\[
C_{n,\alpha}^1\equiv \left[
\hoMLE(x_0)-n^{-1/3} \hat{C}_n(x_0)q(\mathbbm{Z},1-\alpha/2),
\hoMLE(x_0)+n^{-1/3} \hat{C}_n(x_0)q(\mathbbm{Z},1-\alpha/2)
\right],
\]
where $q(\mathbbm{Z},1-\alpha/2)$ is the $(1-\alpha/2)^{th}$ quantile of the distribution $\mathbbm{Z}$. These quantiles have been computed in~\cite{groeneboomwellner:2001}, and we will further use $q(\mathbbm{Z},0.975)=0.998181$.
For simulation purposes, we propose

\[
\hat{C}_n(x_0)=\left(
\frac{4 \hoMLE(x_0)\hoMLE'(x_0)}{\Phi_n(\hb,x_0)}
\right)^{1/3},
\]
where $\Phi_n(\beta,x)$ has been defined in~\eqref{def:Phi_n}, and $\hb$ is the maximum partial likelihood estimator. Lemma 4 in~\cite{lopuhaa_nane1} ensures that $\Phi_n(\hb,\cdot)$ is a strong uniform consistent estimator of $\Phi(\beta_0,\cdot)$. Furthermore, as an estimate for $\lambda_0'(x_0)$, we choose the numerical derivative of $\hoMLE$ on the interval that contains $x_0$, that is, the slope of the segment $[\hoMLE(T_{(m)}),\hoMLE(T_{(m+1)})]$.

Pointwise confidence intervals for $\lambda_0(x_0)$ can also be constructed by making use of Theorem~\ref{theorem_inc}. Let $2\log\xi_n(\theta)$ denote the likelihood ratio for testing $H_0:\lambda_0(x_0)=\theta$ versus $H_1:\lambda_0(x_0)\neq\theta$. A $1-\alpha$ confidence interval is then obtained by inverting the likelihood ratio test $2\log\xi_n(\theta)$ for different values of $\theta$, namely

\[
C_{n,\alpha}^2\equiv
\left\{
\theta: 2\log \xi_n(\theta)\leq q(\mathbbm{D},1-\alpha)
\right\},
\]
where $q(\mathbbm{D},1-\alpha)$ is the $(1-\alpha)^{th}$ quantile of the distribution $\mathbbm{D}$. Quantiles of $\mathbbm{D}$, based on discrete approximations of Brownian motion, are provided in~\cite{banerjeewellner:2005}, and we will make use of $q(\mathbbm{D},0.95)=2.286922$. The parameter $\theta$ is chosen to take values on a fine grid between $0$ and $6$. It can be shown immediately that, for large enough $n$, the coverage probability of $C_{n,\alpha}^2$ is approximately $1-\alpha$.

For the performance analysis, we have constructed and compared, from simulated data, the confidence intervals $C_{n,\alpha}^1$ and $C_{n,\alpha}^2$, for $\alpha=0.05$ and various $n$. We will assume a Weibull baseline distribution function for the event times, with shape parameter $2$ and scale parameter $1$.
For simplicity, we will assume that the covariate is single-valued and uniformly $(0,1)$ distributed and take $\beta_0=0.5$. Given the covariate, the censoring times are assumed to be uniformly $(0,1)$ distributed. We will choose $x_0=\sqrt{\log2}$, the median of the baseline distribution of the event times. For each chosen sample size, we generate $1000$ replicates and compute the empirical coverage and the average length of the corresponding confidence intervals. Furthermore, since we are simulating from a Weibull distribution with shape parameter $2$ and scale parameter $1$, and hence know the true baseline hazard function $\lambda_0$ and its derivative, as well as the true underlying regression coefficient, we could also consider a confidence interval $\bar{C}_{n,\alpha}^1$, given by

\[
\bar{C}_{n,\alpha}^1\equiv \left[
\hoMLE(x_0)- n^{-1/3} C_0(x_0)q(\mathbbm{Z},1-\alpha/2),
\hoMLE(x_0)+n^{-1/3} C_0(x_0)q(\mathbbm{Z},1-\alpha/2)
\right],
\]
where $C_0$ is a deterministic function given by

\[
C_0(x_0)=
\left(
\frac{4v\lambda_0(x_0)\lambda_0'(x_0)}{\Phi(\beta_0,x_0)}
\right)^{1/3}.
\]

Table~\ref{table:sim} reveals the performance, for various sample sizes, of the confidence interval $C_{n,0.05}^2$ based on the likelihood ratio method (LR), the confidence interval $C_{n,0.05}^1$, based on the asymptotic distribution (AD) of the scaled differences between the NPMLE $\hoMLE$ and the true baseline hazard at a fixed point, as well as the confidence interval $\bar{C}_{n,0.05}^1$ based on the known Weibull distribution (TD).

\begin{table}[ht]
\centering
\begin{tabular}{|c|c|c|c|c|c|c|}
  \hline
  &\multicolumn{2}{c|}{LR}&\multicolumn{2}{c|}{AD}&\multicolumn{2}{c|}{TD}\\[0.5ex]
  \hline
  n & AL & CP & AL & CP & AL & CP \\ [0.5ex] \hline
  50 & 4.275  & 0.917 & 5.203 & 0.932 & 1.506 & 0.964 \\ \hline
  100 & 3.837 & 0.923 & 4.838 & 0.941 & 1.317 & 0.953 \\ \hline
  200 & 3.009 & 0.931 & 4.605 & 0.947 & 1.247 & 0.947 \\ \hline
  500 & 2.734 & 0.947 & 3.372 & 0.948 & 0.961 & 0.964 \\ \hline
  1000 & 1.454 & 0.942 & 2.259 & 0.940 & 0.713 & 0.957 \\ \hline
  5000 & 0.879 & 0.945 & 1.768 & 0.952 & 0.546 & 0.953 \\ \hline
\end{tabular}
\caption{Simulaton results for constructing $95\%$ pointwise confidence intervals using the likelihood ratio $C_{n,0.05}^2$ (LR) or the asymptotic distribution of the NPMLE estimator $C_{n,0.05}^1$ (AD) and $\bar{C}_{n,0.05}^1$ (TD), in terms of average length (AL) and empirical coverage (CP).}
\label{table:sim}
\end{table}

\newpage

It is noteworthy that for each sample size, the likelihood ratio method yields, on average, shorter pointwise confidence intervals in comparison with the confidence intervals based on the asymptotic distribution of the NPMLE estimator $\hoMLE$. Moreover, the confidence intervals based on the likelihood ratio exhibit comparable coverage probabilities with the confidence intervals $C_{n,0.05}^2$, based on the asymptotic distribution. As expected, the highest coverage rate is attained by the confidence intervals $\bar{C}_{n,0.05}^1$. Furthermore, they also yield confidence intervals with the shortest length, on average.


\pagebreak
\bibliographystyle{acm}
\bibliography{paper3}
\end{document}